\documentclass[11pt,oneside,reqno]{amsart}

\usepackage{fullpage}

\usepackage{color}
\usepackage{amssymb}
\usepackage{amsmath}
\usepackage{arydshln}
\usepackage{hyperref}
\usepackage{bbm}
\usepackage{mathrsfs}

\usepackage{dsfont}
\usepackage{xcolor}

\usepackage{verbatim} 

\hypersetup{
	colorlinks   = true, 
	urlcolor     = {blue!90!black}, 
	linkcolor    = {blue!90!black}, 
	citecolor   = {red!90!black} 
}

\def\BibTeX{{\rm B\kern-.05em{\sc i\kern-.025em b}\kern-.08em
    T\kern-.1667em\lower.7ex\hbox{E}\kern-.125emX}}

\numberwithin{equation}{section}
\numberwithin{figure}{section}

\newtheorem{satz}{Satz}[section]
\newtheorem{Theorem}[satz]{Theorem}
\newtheorem{Proposition}[satz]{Proposition}

\newtheorem{Corollary}[satz]{Corollary}

\newtheorem{Lemma}[satz]{Lemma}

\newtheorem{Remark}[satz]{Remark}

\newtheorem{Example}[satz]{Example}

\newcommand{\R}{\mathbb{R}}
\newcommand{\E}{\mathbb{E}}
\renewcommand{\P}{\mathbb{P}}

\newcommand{\e}{\varepsilon}
\newcommand{\1}{\mathds{1}}
\newcommand{\F}{\mathcal{F}}
\newcommand{\X}{\mathcal{X}}
\newcommand{\N}{\mathbb{N}}
\newcommand{\Z}{\mathbb{Z}}

\begin{document}

\title[IPS with continuous spins]{Interacting particle systems with continuous spins}

\author[Viktor Bezborodov]{Viktor Bezborodov}
\address[Viktor Bezborodov]{Institute for Mathematical Stochastics,
Georg-August-University of Goettingen,
Goldschmidtstra{\ss}e 7,
37077 Goettingen,
Germany}
\email{viktor.bezborodov@uni-goettingen.de}

\address[Viktor Bezborodov]{Wroc{\l}aw University of Science and Technology,
Faculty of Electronics,
ul. Janiszewskiego 11/17,
50-372 Wroc{\l}aw, Poland}
\email{viktor.bezborodov@pwr.edu.pl}

\author[Luca Di Persio]{Luca Di Persio}
\address[Luca Di Persio]{Department of Computer Science, University of Verona, Strada le Grazie 15, Verona, Italy}
\email{luca.dipersio@univr.it}

\author{Martin Friesen}
\address[Martin Friesen]{School of Mathematical Sciences, Dublin City University, Glasnevin, Dublin 9, Ireland}
\email{martin.friesen@dcu.ie}

\author[Peter Kuchling]{Peter Kuchling}
\address[Peter Kuchling]{Hochschule Bielefeld - University of Applied Sciences and Arts}
\email{peter.kuchling@hsbi.de}

\date{\today}



\begin{abstract} 
 We study a general class of interacting particle systems over a countable state space $V$ where on each site $x \in V$ the particle mass $\eta(x) \geq 0$ follows a stochastic differential equation. We construct the corresponding Markovian dynamics in terms of strong solutions to an infinite coupled system of stochastic differential equations and prove a comparison principle with respect to the initial configuration as well as the drift of the process. Using this comparison principle, we provide sufficient conditions for the existence and uniqueness of an invariant measure in the subcritical regime and prove convergence of the transition probabilities in the Wasserstein-1-distance. Finally, for sublinear drifts, we establish a linear growth theorem showing that the spatial spread is at most linear in time. Our results cover a large class of finite and infinite branching particle systems with interactions among different sites.
\end{abstract}

\maketitle

\allowdisplaybreaks

\section{Introduction}

We consider a continuous-state branching process over a countable space $V$. For an a-priori fixed weight function $v: V \longrightarrow (0,\infty)$, we define the space of tempered configurations over $V$ via
\begin{equation}\label{Linderung = relief}
    \mathcal{X} = \bigg\{ \eta = (\eta(x))_{x \in V} \ | \ \eta(x) \in \R_+\ \forall x \in V \text{ and } \sum_{x \in V}\eta(x)v(x) < \infty \bigg\}. 
\end{equation}
Then $(\X,d)$ is a complete and separable metric space when equipped with the distance given by the weighted sum
\[
 d(\eta,\xi)=\| \eta - \xi \| := \sum_{x \in V}v(x)|\eta(x) - \xi(x)|.
\]
We equip $\X$ with the corresponding Borel $\sigma$-algebra. For a given configuration $\eta = (\eta(x))_{x \in V} \in \X$ the number $\eta(x) \geq 0$ describes the mass of particles at the site $x \in V$. By $0 \in \X$ we denote the empty configuration $\eta(x) = 0$ for all $x \in V$.  The weight function $v$ allows for a flexible treatment of finite and infinite particle systems. Indeed, if $\inf_{x \in V}v(x) > 0$, then elements in $\X$ are necessarily summable sequences that correspond to finite particle configurations. On the other side, if $v(x)$ has sufficient decay at "infinity", then $\X$ may contain sequences that are not summable corresponding to infinite tempered configurations.

In this work, we study particle dynamics on $\X$ where at each moment of time $t\geq 0$ and each $x \in V$ the value of the process $\eta_t(x)$ represents the mass at location $x$ at time $t$. This mass follows the system of stochastic equations
\begin{align}\label{SDE}
    \eta_t(x) &= \eta_0(x) + \int_0^t B(x,\eta_s)ds 
    + \int_0^t \sqrt{2c(x,\eta_s(x))}dW_s(x)
    \\* \notag &\ \ \ + \int_0^t \int_{\X \backslash \{0\}}\int_{\R_+} \nu(x) \mathbbm{1}_{ \{ u \leq g(x,\eta_{s-}(x))\} } \widetilde{N}_x(ds,d\nu,du)
    \\* \notag &\ \ \ + \sum_{y \in V \backslash \{x\}}\int_0^t \int_{\X \backslash \{0\}}\int_{\R_+} \nu(x) \mathbbm{1}_{ \{ u \leq g(y,\eta_{s-}(y))\} } N_y(ds,d\nu,du)
    \\* \notag &\ \ \ + \int_{0}^t \int_{\X \backslash \{0\}} \int_{\R_+}\nu(x) \mathbbm{1}_{ \{ u \leq \rho(x,\eta_{s-},\nu)\} } M(ds,d\nu,du)
\end{align}
where $B(x, \cdot): \X \longrightarrow \R$,
$c(x,\cdot),g(x,\cdot): \R_+ \longrightarrow \R_+$ with $c(x,0) = g(x,0) = 0$, and $\rho: V \times \X \times \X \backslash \{0\} \longrightarrow \R_+$ are measurable functions. All these parameters are supposed to satisfy the additional conditions specified below. The noise terms are defined on a stochastic basis $(\Omega, \mathcal{F}, (\mathcal{F}_t)_{t \geq 0}, \P)$ with the usual conditions and satisfy the following assumptions:
\begin{itemize}
 \item[(N1)] $(W_t(x))_{t\geq 0,\ x \in V}$ are mutually independent one-dimensional $(\mathcal{F}_t)_{t \geq 0}$-Brownian motions.
 \item[(N2)] $(N_y(ds,d\nu,du))_{y \in V}$ are mutually independent $(\mathcal{F}_t)_{t \geq 0}$-Poisson random measures on $\R_+ \times \X \backslash \{0 \} \times \R_+$ with compensator $\widehat{N}_y(ds,d\nu,du) = ds H_1(y,d\nu)du$, 
 where $H_1(y,d\nu)$ is, for each $y \in V$, a sigma-finite measure on $\X \backslash \{0\}$.
 \item[(N3)] $M(ds,d\nu,du)$ is an $(\F_t)_{t \geq 0}$-Poisson random measure on $\R_+ \times \X \backslash \{0 \} \times \R_+$ with compensator $M(ds,d\nu,du) = ds H_2(d\nu)du$, where $H_2(d\nu)$ is a sigma-finite measure on $\X \backslash \{0\}$.
 \item[(N4)] The noise terms $(W_t(x))_{t\geq 0,\ x \in V}$, $(N_y(ds,d\nu,du))_{y \in V}$, $M(ds,d\nu,du)$ are independent.
\end{itemize}
Finally, we let $\widetilde{N}_x = N_x - \widehat{N}_x$ denote the compensated Poisson random measure. For the notion of weak and strong existence, we employ the following standard definition. A strong solution of \eqref{SDE} is an $(\F_t)_{t \geq 0}$-adapted cadlag process $(\eta_t)_{t \geq 0} \subset \X$ such that \eqref{SDE} holds a.s. for each $x \in V$ and $t \geq 0$. A weak solution consists of a stochastic basis $(\Omega, \F, (\F_t)_{t \geq 0}, \P)$, an $(\F_t)_{t \geq 0}$-adapted cadlag process $(\eta_t)_{t \geq 0} \subset \X$ and noise terms (N1) -- (N4) such that \eqref{SDE} holds a.s. for each $x \in V$ and $t \geq 0$. In this definition, we implicitly assume that all integrals in \eqref{SDE} are well-defined. 

The system of stochastic equations \eqref{SDE} contains finite and infinite dimensional models studied in the literature. In the finite-dimensional case, $V = \{1,\dots, m\}$ with $v(x) = 1$, we have $\X = \R_+^m$. If the coefficients $B(x,\eta), c(x,t), g(x,t)$ are affine linear in $\eta$ and $t$, and $\rho = 1$, then \eqref{SDE} reduces to multi-type continuous-state branching processes with immigration as constructed in \cite{BLP15}, see also e.g. \cite[Chapter 3]{L11} for the one-dimensional case $m = 1$. Moreover, for general $B,c,g$ and $V = \{1\}$ equation \eqref{SDE} reduces to nonlinear continuous-state branching processes with immigration studied in \cite{FL10}, \cite{LYZ19}, and \cite{FJKR19}. Thus, the stochastic particle system studied in this work provides an infinite-dimensional extension of multi-type CBI processes and their non-linear analogues. More generally, for infinite state spaces $V$, our model covers a wide class of interacting particle systems including, e.g., multi-type branching systems \cite{P06, F22}, population models with interactions \cite{HW07}, systems of particles driven by $\alpha$-stable noises \cite{XZ10}, branching random walks with discrete state space \cite{ShiBook}. For other related literature, we refer to \cite{DG03}.

Here and below we write, for $\eta,\xi \in \X$, $\eta \leq \xi$ if $\eta(x) \leq \xi(x)$ holds for all $x \in V$. We impose the following conditions on the coefficients of \eqref{SDE}:
\begin{enumerate}
    \item[(A1)] The drift coefficient $B(x,\eta)$ has the form $B(x,\eta) = B_0(x,\eta) - B_1(x,\eta(x))$ where $B_0(x,\cdot): \X \longrightarrow \R_+$ and $B_1(x,\cdot): \R_+ \longrightarrow \R_+$ are measurable mappings for each $x \in V$. For each $R > 0$ there exists a constant $C_1(R) \geq 0$ such that
   \[
   \|B_0(\cdot, \eta) - B_0(\cdot, \xi) \| \leq C_1(R) \| \eta - \xi \|
   \]
   holds for all $\eta,\xi \in \X$ with $\| \eta\|, \| \xi \| \leq R$. The function $\R_+ \ni t \longmapsto B_1(x,t)$ is continuous, non-decreasing, and $B_1(x,0) = 0$ holds for each $x \in V$. Finally, for all $\eta, \xi \in \X$ satisfying $\eta \leq \xi$, it holds that 
   \[
    B_0(x, \eta) \leq B_0(x,\xi), \qquad \forall x \in V.
   \]
   
   \item[(A2)] For each $x \in V$ there exists a constant $C_2(x) \geq 0$ such that 
   \[
    |c(x,t) - c(x,s)| \leq C_2(x)|t-s|
   \]
   holds for  $t,s \in \R_+$, and
   \[
    \sum_{x \in V}v(x)C_2(x) < \infty.
   \]
    Furthermore $c(x,0) = 0$ for $x \in V$.

   \item[(A3)] For each $x \in V$ there exists a constant $C_3(x) \geq 0$
   such that 
   \[
    |g(x,t) - g(x,s)| \leq C_3(x)|t-s|, \qquad t,s \geq 0.
   \]
   The function $\R_+ \ni t \longmapsto g(x,t)$ is non-decreasing for each $x \in V$, and one has $g(x,0) = 0$ for each $x \in V$.
    
   \item[(A4)] It holds that
   \[
     \sum_{x \in V}v(x)C_3(x)\int_{ \{ \| \nu \| \leq 1\} \backslash \{0\}} \nu(x)^2 H_1(x,d\nu) < \infty, 
    \]
   and, there exists a constant $C_4 \geq 0$ such that
   \[
    C_3(x) \int_{ \{\| \nu \| > 1 \}} \nu(x) H_1(x,d\nu) \leq C_4
   \]
   and 
   \[
     C_3(x)\int_{\X \backslash \{0\}} \sum_{y \in V \backslash \{x\}} v(y)\nu(y) H_1(x,d\nu) \leq C_4 v(x)
   \]
   hold for all $x \in V$.
 
   \item[(A5)] For each $R > 0$ there exists a constant $C_5(R) \geq 0$ such that
    \[
     \int_{\X \backslash \{0\}} \sum_{x \in V}v(x)\nu(x) | \rho(x,\eta, \nu) - \rho(x,\xi, \nu)| H_2(d\nu) \leq C_5(R)\| \eta - \xi\|
    \]
    holds for all $\eta,\xi \in \X$ with $\| \eta \|, \|\xi\| \leq R$. For all $\eta, \xi \in \X$ satisfying $\eta \leq \xi$, it holds that
    \[
     \rho(x,\eta,\nu) \leq \rho(x,\xi,\nu), \qquad \forall x \in V, \ \ \nu \in \X \backslash \{0\}.
    \]
    
  \item[(A6)] There exists a constant $C_6 \geq 0$ such that 
  \[
   \| B_0(\cdot, \eta)\| + \int_{\X \backslash \{0\} } \sum_{x \in V}v(x)\nu(x) \rho(x,\eta,\nu) H_2(d\nu) \leq C_6(1 + \| \eta\|).
  \]
\end{enumerate}
We say that the tuple $(B, B_0, B_1, c, g, \rho)$
is $C_{\overline{1,6}}$-admissible if conditions (A1)--(A6) are satisfied with given $C_1, C_2, C_3, C_4, C_5, C_6$. Sufficient conditions for these assumptions and particular examples are discussed in the next section. Under this general set of conditions, we derive the following existence and uniqueness result of strong solutions, as well as the comparison property of solutions.  
\begin{Theorem}\label{thm existence uniqueness comparison}
 Suppose that conditions (A1) -- (A6) are satisfied. Then for each $\F_0$-measurable random variable $\eta_0$ with $\E[\|\eta_0\|] < \infty$ there exists a unique strong solution $(\eta_t)_{t \geq 0}$ in $\X$ of \eqref{SDE}. Moreover, there exists a constant $C > 0$ independent of $\eta_0$ such that
 \begin{align*}
     \E[ \| \eta_t \| ] \leq \left(1 + \E[ \| \eta_0 \| ]\right) e^{Ct}, \qquad t \geq 0.
 \end{align*}
 Finally, for any $\eta_0, \xi_0 \in \X$ with $\E[\|\eta_0\|], \ \E[ \|\xi_0\|] < \infty$ let $(\eta_t)_{t \geq 0}, (\xi_t)_{t \geq 0}$ be the unique strong solutions of \eqref{SDE}. If $\P[ \xi_0(x) \leq \eta_0(x), \ \ \forall x \in V] = 1$, then 
 \[ 
  \P[\xi_t \leq \eta_t, \ \forall t \geq 0] = 1. 
 \]
\end{Theorem}
This theorem will be deduced from the results of Sections 3 - 6. Indeed, in Section 3 we prove the non-explosion and first moment bound for any solution of \eqref{SDE}. In Section 4 we establish the pathwise uniqueness of solutions under slightly weaker conditions than (A1) -- (A6), while the comparison principle is derived in Section 5. To prove these results, we provide an infinite dimensional extension of the classical Yamada-Watanabe theorem, see also \cite{FL10, Ma13, BLP15, FJR19b} for some finite-dimensional results in this direction. Finally, in Section 6 we prove, by using finite-dimensional approximations combined with the comparison principle, the weak existence of solutions of \eqref{SDE}. Combining all these results gives in view of the Yamada--Watanabe--Engelbert theorem (see \cite{MR2336594}) the strong existence of a unique solution of \eqref{SDE}. 

The space $\X$ introduced in \eqref{Linderung = relief} can be seen as a non-negative cone in an $L^1$-type space. The space of configurations summable with respect to given weights is a natural choice for a state space in the construction of interacting particle systems. This choice goes back at least to Liggett and Spitzer \cite{LS81} and Andjel \cite{Andj82}. The construction of the stochastic particle systems as solutions to stochastic equations driven by Poisson point processes is not uncommon, as it was used in \cite{Gar95, GK06, BaDlattice} for the study of birth-and-death processes with an infinite number of particles. Such stochastic equations can be seen as a natural development of the  graphical construction for classical interacting particle systems
such as the contact process or the voter model \cite{Lig99}.

As an application of this construction and the comparison principle with respect to the initial conditions, we prove under a suitable subcriticality condition on the drift the existence and uniqueness of and convergence towards the invariant measure in the Wasserstein distance. Let $\mathcal{P}(\X)$ be the space of all Borel probability measures over $\X$ and let $\mathcal{P}_1(\X)$ be the subspace of measures with finite first moment, i.e. $\int_{\X}\|\nu\| \varrho(d\nu) < \infty$. Let 
\begin{equation}\label{eq:transition_prob}
 p_t(\eta,d\xi) = \P[ \eta_t \in d\xi \ | \ \eta_0 = \eta]
\end{equation}
be the transition probabilities of the process $(\eta_t)_{t \geq 0}$ obtained from \eqref{SDE}. A general version of Yamada--Watanabe theorem and Theorem \ref{thm existence uniqueness comparison} imply that the strong solution is unique in law (\cite[Theorem 3.14]{MR2336594}); therefore $ p_t$ in \eqref{eq:transition_prob} is well defined. Define the semigroup $(P^*_t)_{t \geq 0}$ by $P^*_t \rho = \int_{\X} p_t(\eta, \cdot) \rho(d\eta)$, where $\rho \in \mathcal{P}_1(\X)$. Recall that $\pi \in \mathcal{P}(\X)$ is called \textit{invariant measure} if $P^*_t \pi =  \pi$ for $t > 0$. The Wasserstein-1 distance is defined by
\[
 W_1(\varrho, \widetilde{\varrho}) = \inf_{G \in \mathcal{H}(\varrho, \widetilde{\varrho})}\left\{ \int_{\X \times \X}\| \eta - \xi\| G(d\eta,d\xi) \right\}
\]
where $\mathcal{H}(\varrho, \widetilde{\varrho})$ denotes the set of all couplings of $(\varrho, \widetilde{\varrho})$ on the product space $\X \times \X$. Define the effective drift
\begin{align}\label{eq: effective drift}
 \widetilde{B}(x,\eta) &= B(x,\eta) + \sum_{y \in V \backslash \{x\}}g(y,\eta(y))\int_{\X \backslash \{0\}}\nu(x)H_1(y,d\nu)
 \\ \notag &\ \ \ + \int_{\X \backslash \{0\}}\nu(x)\rho(x,\eta,\nu)H_2(d\nu).
\end{align}
Note that this function is well-defined due to conditions (A4) and (A6). 

\begin{Theorem}\label{thm: invariant measures}
 Suppose that conditions (A1) -- (A6) are satisfied and that the constants from conditions (A1) and (A5) satisfy $\sup_{R > 0}(C_1(R) + C_5(R)) < \infty$. Assume additionally that there exists $A > 0$ such that 
 \begin{align}\label{eq: subcritical drift}
  \sum_{x \in V}v(x)\left(\widetilde{B}(x,\eta) - \widetilde{B}(x,\xi)\right) 
  \leq - A\| \eta - \xi\| 
 \end{align} 
 holds for all $\eta,\xi \in \X$ with $\xi \leq \eta$. Then for any $\varrho, \widetilde{\varrho} \in \mathcal{P}_1(\X)$ one has
 \begin{align}\label{eq: wasserstein 1 contraction estimate}
  W_1(P_t^*\varrho, P_t^*\widetilde{\varrho}) \leq e^{-At}W_1(\varrho, \widetilde{\varrho}), \qquad t \geq 0.
 \end{align}
 In particular, there exists a unique invariant measure $\pi \in \mathcal{P}_1(\X)$ and it holds that
 \begin{align}\label{eq: wasserstein 1 inv measure}
  W_1(P_t^*\varrho, \pi) \leq e^{-At}W_1(\varrho, \pi), \qquad t \geq 0.
 \end{align}
\end{Theorem}

This result extends the methods from \cite{FJR19a} and \cite{FJKR19} to the infinite-dimensional case of particle systems on $\X$. While the comparison principle still remains the key tool to derive stability estimates in the $L^1$-norm, when working with infinite-dimensional settings, additional conditions are required in order to control the states for all sites $x \in V$. The latter are reflected by the assumption $\sup_{R > 0}(C_1(R) + C_5(R)) < \infty$. Assumption \eqref{eq: subcritical drift} is motivated by the one-dimensional case studied in \cite{FJKR19} and can be viewed as a subcriticality (or strong dissipativity) condition on the drift, i.e., the drift needs to be sufficiently negative. Without immigration, such a condition implies that the invariant measure is given by the empty configuration $\delta_0$. In the presence of non-trivial immigration, the invariant measure is necessarily non-trivial as well. 

It is worth mentioning that for branching systems with interactions such a condition is certainly not optimal as recently was demonstrated in \cite{LLWZ22} for the one-dimensional case. However, our conditions are relatively simple to verify and allow for a simple illustrative proof. The extension of \cite{LLWZ22} to our infinite dimensional case is yet an open problem from the literature. Finally, our subcriticality condition \eqref{eq: subcritical drift} also rules out the possibility of multiple invariant measures, a fact that is natural for certain infinite particle systems as studied in \cite{HW07}. An extension of our results to such types of settings is, however, beyond the scope of this work.

In the last part of this work, we investigate the growth of a finite particle system without immigration (that is $H_2(d\nu) = 0)$ when started at a single point. We demonstrate that whenever the effective drift $\widetilde{B}$ defined in \eqref{eq: effective drift} is sublinear,  the spatial spread of the process is at most linear in time. Also, we provide a super-exponential bound in the 'large deviations' region of the process. For both results, we suppose that $V$ is the vertex set of an infinite connected graph $G = (V, E)$ of bounded degree. We let $\mathrm{dist}(z,z')$ be the graph distance for $z,z' \in V$, and denote by $\mathbb{B}(x_0, r )\subset V$ the closed ball of radius $r$ with respect to this graph distance. 

Let us underline that, unlike particle systems with discrete states, determining what qualifies as an `occupied site' for a continuous-state process is not a straightforward task. In the following result, we interpret a site as occupied if the particle mass is larger than a given threshold $\e > 0$.
\begin{Theorem}\label{thm: at most linear}
Suppose conditions (A1) -- (A6) are satisfied with $H_2(d\nu) = 0$ and weight function $v$. Let $x_0 \in V$ and let $(\eta_t)_{t \geq 0}$ be the unique solution of \eqref{SDE} with initial condition $\eta_0(x) = \1_{\{ x = x_0 \}}$. Assume that there exists bounded $b: V \times V \longrightarrow \R_+$ such that, for all $x \in V$ and $\eta \in \X$, one has
\begin{equation}\label{Wirrsal = confusion}
     \widetilde{B}(x,\eta) \leq \sum_{y \in V}b(x,y)\eta(y), 
\end{equation}
there exists $R \in \N$ such that $b(x,y) = 0$ holds for $x, y \in V$ satisfying $\mathrm{dist}(x,y) > R$, and the weight function $v$ satisfies 
\begin{align}\label{eq: v growth condition}
 \sup_{\mathrm{dist}(x,y) \leq R}\frac{v(y)}{v(x)} < \infty.
\end{align} 
Then there exist constants $C, c, m > 0$ such that, for all $x \in V$ and $t \geq 0$, one has
 \begin{equation} \label{der Vorrat = supply}
      \E \left[\sup_{r \in [0,t]} \eta_r(x) \right] \leq C \exp\left[- c{ \text{ \emph{dist}}}(x_0, x) \ln (  \text{\emph{dist}}(x_0, x) ) + mt\right].
 \end{equation}
 Moreover, for any $\varepsilon > 0$  we find $\mathrm{C}_l > 0$ such that 
 \begin{equation}\label{in Auftrag geben}
     \left\{z \in V : \sup_{r \in [0,t]} \eta_r (z) \geq \varepsilon \right\} \subset \mathbb{B}(x_0, \mathrm{C}_l t )
 \end{equation}
 holds a.s. for  $t \geq t_0$, where $t_0$ is random.  The constant $\mathrm{C}_l$ may depend only on $\varepsilon$ and the parameters of the 
 process. 
 \end{Theorem}
 
 The assumption on the effective drift $\widetilde{B}$ allows us to compare the process with a simpler process that consists of a linear drift plus a martingale part. The condition on $b$ essentially states that the branching mechanism of this process has finite range. Note that \eqref{eq: v growth condition} implies that $v$ satisfies with $e^\kappa = \sup_{\mathrm{dist}(x,y) \leq R}\frac{v(y)}{v(x)}$ the growth bound 
 \begin{align}\label{v: growth bound 2}
  v(x_0)e^{-\kappa \mathrm{dist}(x_0,x)} \leq v(x) \leq v(x_0)e^{\kappa\mathrm{dist}(x_0,x)}, \qquad x \in V.
 \end{align}
  
 The proof of Theorem \eqref{thm: at most linear} is given in the last section of this work. Our proof relies on the comparison principle combined with heat kernel estimates of a random walk on the graph $G$. Namely, using the comparison principle combined with the linear growth condition on the effective drift, we obtain a bound of the form $\eta_t \leq \zeta_t$, where $\zeta_t$ has a constant drift. It is easy to see that $\eta_t$ satisfies the assertion whenever the larger process $\zeta_t$ satisfies the assertion. To prove that $\zeta_t$ satisfies the above theorem, we use an auxiliary graph to apply known heat kernel estimates from \cite[Corollary 12]{HeatKernelGraph}.

This work is organized as follows. In Section 2 we discuss particular examples of \eqref{SDE} and further elaborate on related literature. Section 3 is devoted to the non-explosion and first moment bound on solutions of \eqref{SDE}. Pathwise uniqueness is established in Section 4 while the comparison principle is proven in Section 5. Section 6 contains the proof of the weak existence of solutions of \eqref{SDE} while the proof of Theorem \ref{thm: invariant measures} is proven in Section 7. Finally, the linear spread, that is Theorem \ref{thm: at most linear}, is proven in Section 8. A few minor technical results are given in the appendix.

\section{Sufficient conditions and particular examples}

As a first step, we state a proposition that allows us to verify conditions (A1) -- (A6) in a general framework and hence serves as a toolbox for particular examples. 
\begin{Proposition}\label{prop: sufficient conditions}
 Let $v: V \longrightarrow (0,\infty)$ and suppose that
 \begin{enumerate}
     \item[(i)] The drift $B: V \times \X \longrightarrow \R$ is given by 
    \[
    B(x,\eta) = \left( b(x) + \sum_{y \in V} a(x,y)\eta(y)\right) - m(x)\eta(x)^{\lambda}
    \]
    where $\lambda \geq 0$, $b, m: V \longrightarrow \R_+$, $\|b\| < \infty$, and $a: V \times V \longrightarrow \R$ is such that $a(x,y) \geq 0$ for $x \neq y$, and there exists $C_1 > 0$ satisfying
    \begin{align}\label{eq: a drift}
    \sum_{y \in V \backslash \{x\}}v(y)a(x,y) + \1_{\{a(x,x) \geq 0\}}a(x,x)v(x) \leq C_1v(x), \qquad x \in V.
    \end{align}

  \item[(ii)] The diffusion coefficient is given by $c(x,\eta(x)) = c(x)\eta(x)$ where $c: V \longrightarrow \R_+$ satisfies $\| c\| < \infty$.

  \item[(iii)] The branching rate is given by $g(x,\eta(x)) = g(x)\eta(x)$ with $g: V \longrightarrow \R_+$, and $H_1(x,\cdot)$ is a family of $\sigma$-finite measures on $\X \backslash \{0\}$ such that
  \begin{align*}
        &\ \sum_{x \in V}v(x)g(x) \int_{\{\|\nu\| \leq 1\}\backslash \{0\}} \nu(x)^2 H_1(x,d\nu) < \infty
    \end{align*}
    and
    \begin{align*}
        \\ &\ \sup_{x \in V}g(x)\int_{\{\|\nu\|>1\}} \nu(x)H_1(x,d\nu) 
        + \sup_{x \in V}\frac{g(x)}{v(x)} \int_{\X \backslash \{0\}} \sum_{y \in V \backslash \{x\}}v(y)\nu(y) H_1(x,d\nu) < \infty.
  \end{align*}

  \item[(iv)] The immigration rate-function $\rho: V \times \X \times \X \longrightarrow \R_+$ is given by
  \[
  \rho(x, \eta, \nu) = \sum_{y \in V}\varphi(x,y)\eta(y) + \sum_{y \in V}\psi(x,y)\nu(y)
  \]
  with $\varphi, \psi: V \times V \longrightarrow \R_+$, and $H_2(d\nu)$ is a $\sigma$-finite measure on $\X \backslash \{0\}$ such that
  \begin{align*}
   &\ \sup_{y \in V}\frac{1}{v(y)}\int_{\X \backslash \{0\}} \sum_{x \in V}v(x)\nu(x)\varphi(x,y)H_2(d\nu) < \infty,
   \\ &\ \sum_{y \in V}\int_{\X \backslash \{0\}} \sum_{x \in V}v(x)\nu(x)\psi(x,y)\nu(y) H_2(d\nu) < \infty.
  \end{align*}
 \end{enumerate}  
 Then conditions (A1) -- (A6) are satisfied.
 \end{Proposition}
\begin{proof} 
 Let us write $B(x,\eta) = B_0(x,\eta) - B_1(x,\eta(x))$ with
 \begin{align*}
     B_0(x,\eta) &= b(x) + \sum_{y \in V \backslash \{x\}}a(x,y)\eta(y) + \1_{\{a(x,x)\geq 0\}}a(x,x)\eta(x),
     \\ B_1(x,\eta(x)) &= \1_{\{a(x,x) < 0\}}|a(x,x)|\eta(x) + m(x)\eta(x)^{\lambda}.
 \end{align*}
 Then $B_1$ has the desired properties stated in condition (A1) while $B_0$ satisfies by \eqref{eq: a drift}
 \begin{align}\label{eq: B drift}
  \| B_0(\cdot,\eta)\| \leq \max\{\|b\|,\ C_1 \}(1 + \|\eta\|) \ \text{ and }\ 
  \| B_0(\cdot,\eta) - B_0(\cdot,\xi)\| \leq C_1 \|\eta - \xi\|,
 \end{align}
 and $B_0(x,\eta) \leq B_0(x,\xi)$ for each $x \in V$ whenever $\eta \leq \xi$.  This shows that (A1) is satisfied with $C_1(R) = C_1$. It is straightforward to verify conditions (A2) with $C_2(x) = c(x)$ and (A3) with $C_3(x) = g(x)$. Condition (A4) follows directly from (iii) with 
 \[
  C_4 = \sup_{x \in V}\max\left\{ g(x)\int_{\{\|\nu\|>1\}} \nu(x)H_1(x,d\nu), \ \frac{g(x)}{v(x)} \int_{\X \backslash \{0\}} \sum_{y \in V \backslash \{x\}}v(y)\nu(y) H_1(x,d\nu)\right\}. 
 \]
 Concerning assumption (A5) let us note that
 \begin{align*}
  |\rho(x,\eta, \nu) - \rho(x,\xi, \nu)| \leq \sum_{y \in V}\varphi(x,y)|\eta(y) - \eta(y)|
 \end{align*}
 and hence
 \begin{align*}
  &\ \int_{\X \backslash \{0\}}\sum_{x \in V} v(x)\nu(x)|\rho(x,\eta, \nu) - \rho(x,\xi, \nu)| H_2(d\nu)
  \\ &\leq \sum_{y \in V}|\eta(y) - \eta(y)| \int_{\X \backslash \{0\}}\sum_{x \in V} v(x)\nu(x) \varphi(x,y)H_2(d\nu)
  \leq C_5 \|\eta - \xi\|
 \end{align*}
 where $C_5 = \sup_{y \in V}\frac{1}{v(y)}\int_{\X \backslash \{0\}}\sum_{x \in V}v(x)\nu(x)\varphi(x,y) H_2(d\nu)$. Finally, condition (A6) is satisfied with
 \begin{align*}
  C_6 &= \max \bigg\{\|b\|,\ C_1,\ \sup_{y \in V}\frac{1}{v(y)}\int_{\X \backslash \{0\}} \sum_{x \in V}v(x)\nu(x)\varphi(x,y)H_2(d\nu), 
  \\ &\qquad \qquad \qquad \qquad \qquad \sum_{y \in V}\int_{\X \backslash \{0\}} \sum_{x \in V}v(x)\nu(x)\psi(x,y)\nu(y) H_2(d\nu)  \bigg\}
 \end{align*}
 due to \eqref{eq: B drift} and
 \begin{align*}
  \int_{\X \backslash \{0\}} \sum_{x \in V}v(x) \nu(x) \rho(x,\eta,\nu)H_2(d\nu)
  &\leq \sum_{y \in V}\eta(y) \int_{\X \backslash \{0\}} \sum_{x \in V}v(x)\nu(x) \varphi(x,y)H_2(d\nu)
  \\ &\qquad + \int_{\X \backslash \{0\}}\sum_{x,y \in V} v(x)\nu(x)\psi(x,y)\nu(y)H_2(d\nu)
  \\ &\leq C_6 (1 + \|\nu\|).
 \end{align*}
\end{proof}

The next remark illustrates specific cases when the inequality \eqref{eq: a drift} is satisfied. It follows the scheme provided in \cite{LS81}.
\begin{Remark}
 Let $V = \Z^d$ be equipped with the $1$-norm $|\cdot|_1$. Let $a: V \times V \longrightarrow \R$ and $v$ be given by one of the following cases: \begin{enumerate}
     \item[(i)] $v(x) = e^{-\delta|x|_1}$ and $a(x,y) = c e^{-\e|x-y|_1}$ for $x \neq y$ with $c > 0$ and $0 < \delta < \e$,
     \item[(ii)] $v(x) = e^{-\delta|x|_1}$ and $a(x,y) = c\1_{\{|x-y|_1\leq R\}}$ for $x \neq y$ with $c,R,\delta > 0$,
     \item[(iii)] $v(x) = \frac{1}{1+|x|_1^{\delta}}$ and $a(x,y) = \frac{c}{1+ |x-y|_1^{\e}}$ for $x \neq y$ with $c > 0$ and $d < \delta < \e$.
 \end{enumerate}
 Then there exists a constant $C_1 > 0$ such that \eqref{eq: a drift} holds. 
\end{Remark}
Note that all these examples satisfy condition \eqref{eq: v growth condition}.
Under the conditions of the previous proposition, one may check that the effective drift is given by
\begin{align*}
    \widetilde{B}(x,\eta) &= \widetilde{b}(x) + \sum_{y \in V}\widetilde{a}(x,y)\eta(y) - m(x)\eta(x)^{\lambda}
\end{align*}
where $\widetilde{b}: V \longrightarrow \R_+$ and $\widetilde{a}: V \times V \longrightarrow \R$ are given by
\begin{align*}
    \widetilde{b}(x) &= b(x) + \sum_{y \in V}\psi(x,y)\int_{\X \backslash \{0\}}\nu(x)\nu(y)H_2(d\nu),
    \\ \widetilde{a}(x,y) &= a(x,y) + \1_{\{x \neq y\}}g(y)\int_{\X \backslash \{0\}}\nu(x)H_1(y,d\nu) + \varphi(x,y)\int_{\X \backslash \{0\}}\nu(x)H_2(d\nu).
\end{align*}
Hence \eqref{eq: subcritical drift} is satisfied, provided that $\lambda = 1$ and $\inf_{x \in V}m(x) > C_1$ where $C_1$ is the constant from \eqref{eq: a drift}. Using the previous remark, we may verify such a condition for particular classes of weight functions $v$ and kernels $a$. 

Formulation \eqref{SDE} contains classical multi-type continuous-state branching processes with immigration as a particular case (when $V$ is finite). The next example shows that it also contains their infinite-dimensional analogues as studied in \cite{MR3849816, F22} in terms of Laplace transforms.

\begin{Example}[infinite-type continuous-state branching process with immigration]\label{Example: CBI}
 Assume conditions (i) -- (iii) of Proposition \ref{prop: sufficient conditions} with $m(x) = 0$, and $\rho(x,\eta,\nu) = 0$ with $H_2$ a $\sigma$-finite measure on $\X \backslash \{0\}$ such that $\int_{\X \backslash \{0\}}\|\nu\|H_2(d\nu) < \infty$. Then conditions (A1) -- (A6) are satisfied. The corresponding process is an infinite-type continuous-state branching process with immigration where $V$ denotes the countable set of different types of the population. 
 \begin{enumerate}
     \item[(a)] If there exists, in addition, a constant $A > 0$ such that 
 \[
  \sum_{x \in V}\widetilde{a}(x,y)v(x) \leq - Av(y), \qquad y \in V,
 \]
 then Theorem \ref{thm: invariant measures} is applicable and the process converges to the unique invariant distribution.
    
    \item[(b)] Suppose that $V = \Z^d$ is equipped with the $1$-norm $|\cdot|_1$. If $b(x) = 0$, $H_2(d\nu) = 0$, $v$ satisfies \eqref{eq: v growth condition} with $\mathrm{dist}(x_0,x) = |x|_1$, and there exists $R > 0$ such that $a(x,y) = 0$ holds for $|x-y|_1 > R$, then Theorem \ref{thm: at most linear} is applicable and the process has at most linear growth.
 \end{enumerate} 
\end{Example}

Below we extend this setting to processes with interactions. For the sake of simplicity, we restrict our attention towards cylindrical branching and immigration measures $H_1, H_2$, which constitutes a natural assumption when $V$ contains infinitely many sites. 
\begin{Remark}
 Suppose that the family of measures $(H_1(x,d\nu))_{x \in V}$ and $H_2(d\nu)$ on $\X \backslash \{0\}$ are given by 
  \begin{align*}
  H_1(x,d\nu) &= \int_{(0,\infty)} \delta_{z \delta_x}(d\nu)\mu_{x,x}(dz) + \sum_{y \in V \backslash \{x\}} \int_{(0,\infty)} \delta_{z \delta_y}(d\nu)\mu_{x,y}(dz),
  \\ H_2(d\nu) &= \sum_{x \in V} \int_{(0,\infty)} \delta_{z\delta_x}(d\nu) \sigma_x(dz)
 \end{align*}
 where $(\mu_{x,y})_{x,y \in V}$ and $(\sigma_x)_{x \in V}$ are L\'evy measures on $(0,\infty)$ satisfying 
 \begin{align*}
  &\ \sum_{x \in V}v(x)g(x) \int_{(0,1]}z^2 \mu_{x,x}(dz) < \infty,
  \\ &\ \sup_{x \in V}g(x)\int_{(1,\infty)}z \mu_{x,x}(dz) + \sup_{x \in V}\frac{g(x)}{v(x)} \sum_{y \in V \backslash \{x\}}v(y) \int_{(0,\infty)}z \mu_{x,y}(dz) < \infty.
 \end{align*}
 Then condition (iii) of Proposition \ref{prop: sufficient conditions} is satisfied. Moreover, if 
  \begin{align*}
  &\ \sup_{y \in V}\frac{1}{v(y)}\sum_{x \in V}v(x) \varphi(x,y)\int_{(0,\infty)}z \sigma_x(dz) < \infty
  \\ &\ \sum_{x \in V}v(x) \int_{(0,\infty)}z \sigma_x(dz) + \sum_{x \in V}v(x)\psi(x,x) \int_{(0,\infty)}z^2 \sigma_x(dz) < \infty
 \end{align*}
 hold, then also condition (iv) of Proposition \ref{prop: sufficient conditions} is satisfied.
\end{Remark}
\begin{proof}
 Let us remark that
    \[
    \int_{ \{ \| \nu \| \leq 1\} \backslash \{0\}} \nu(x)^2 H_1(x,d\nu) 
    = \int_{(0,1]} z^2 \mu_{x,x}(dz) < \infty
    \]
    and
    \[
    \int_{ \{\| \nu \| > 1 \}} \nu(x) H_1(x,d\nu)
    = \int_{(1,\infty)} z \mu_{x,x}(dz) < \infty
    \]
    Moreover, it is easy to see that   
    \[
    \int_{\X \backslash \{0\}} \sum_{y \in V \backslash \{x\}} v(y)\nu(y) H_1(x,d\nu) 
    = \sum_{y \in V \backslash \{x\}} v(y) \int_{(0,\infty)}z \mu_{x,y}(dz).
    \]
    This shows that condition (iii) of Proposition \ref{prop: sufficient conditions} is satisfied. Condition (iv) therein follows from 
    \begin{align*}
    \int_{\X \backslash \{0\}} \sum_{x \in V}v(x)\nu(x)\varphi(x,y)H_2(d\nu)
    \leq \sum_{w \in V}v(w)\varphi(w,y)\int_{\X \backslash \{0\}}z \sigma_w(dz)
    \leq C_{\rho}v(y)
    \end{align*}
    for some constant $C_{\rho} > 0$, and similarly
    \begin{align*}
     &\ \sum_{y \in V}\int_{\X \backslash \{0\}}\sum_{x \in V}v(x)\nu(x)\psi(x,y)\nu(y)H_2(d\nu)
     \\ &= \sum_{w \in V}\int_{(0,\infty)}v(w)\psi(w,w)\int_{(0,\infty)}z^2 \sigma_w(dz) < \infty.
    \end{align*}
\end{proof}

Our next example provides an extension of the infinite-type continuous-state branching process from Example \ref{Example: CBI} towards local interactions in the drift. It extends, in particular, \cite{MR4130578} to the infinite-dimensional case.
\begin{Example}[Local branching process with local competition]
 Let $V = \Z^d$. The continuous state branching Brownian motion on $\X$ is given by the strong solution of 
 \begin{align*}
     d\eta_t(x) &= \left( \sum_{y \in V} a(x,y)\eta_t(y) - m(x)\eta(x)^{\lambda}\right)dt + \sqrt{c(x)\eta_t(x)}dB_t(x) 
     \\* &\qquad + \left(g(x)\eta_t(x)\right)^{\frac{1}{\alpha(x)}}dZ_t(x) + dJ_t(x)
 \end{align*}
 where $(B_t(x))_{t \geq 0}$ is family of independent one-dimensional Brownian motions, $(J_t(x))_{t \geq 0}$ is a family of independent L\'evy subordinators on $\R_+$ with L\'evy measures $\sigma_x$ and drift $b(x) \geq 0$, and $(Z_t(x))_{t \geq 0}$ is a family of independent spectrally positive pure-jump L\'evy processes with L\'evy measure 
 \[
  \mu_{\alpha(x)}(dz) = \1_{(0,\infty)}(z) f(x)\frac{dz}{z^{1+\alpha(x)}}
 \]
 where $\alpha(x) \in (1,2)$ and normalization constant 
 \[
  f(x) = \int_0^{\infty}\left(e^{-z} - 1 + z\right)z^{-1-\alpha(x)}dz = \frac{\Gamma(2 - \alpha(x))}{\alpha(x)(\alpha(x)-1)}.
 \]
 By letting $H_1, H_2$ be given as in previous remark with $\mu_{x,x} = \mu_{\alpha(x)}$, $\mu_{x,z} = 0$, and $\rho(x,\eta,\nu) \equiv 1$, it is easy to see that this model is equivalent in law to \eqref{SDE}. Assume $b, \lambda, m \geq 0$, that $a: \Z^d \times \Z^d \longrightarrow \R$ satisfies $a(x,y) \geq 0$ for $x \neq y$, \eqref{eq: a drift} holds, that 
 \begin{align}\label{eq: alpha 1}
  \sum_{x \in \Z^d}v(x)\left(b(x) + c(x) + \int_{(0,\infty)}z \sigma_x(dz) \right) < \infty,
 \end{align}
 and 
 \begin{align}\label{eq: alpha 2}
  \sum_{x \in \Z^d}v(x)g(x)\frac{\Gamma(2 - \alpha(x))}{\alpha(x)(\alpha(x)-1)(2-\alpha(x))} < \infty, \qquad 
  \sup_{x \in \Z^d} g(x) \frac{\Gamma(2 - \alpha(x))}{\alpha(x)(\alpha(x)-1)^2} < \infty.
 \end{align}
 Then conditions (A1) -- (A6) are satisfied.
 \begin{enumerate}
     \item[(a)] If $\lambda = 1$ and there exists $A > 0$ such that 
     \[
      \sum_{x \in V}a(x,y)v(x) \leq - Av(y) + m(y)v(y)
     \]
     then Theorem \ref{thm: invariant measures} is applicable and the process converges to its unique limit distribution. 
     
     \item[(b)] If $b(x) \equiv 0$, $\sigma_x \equiv 0$, $v$ satisfies \eqref{eq: v growth condition} with $\mathrm{dist}(x_0,x) = |x|_1$, and there exists $R > 0$ such that $a(x,y) = 0$ for $|x-y|_1 > R$, then Theorem \ref{thm: at most linear} is applicable. 
 \end{enumerate}
\end{Example}
\begin{proof}
    Let us show that it is a particular case of Proposition \ref{prop: sufficient conditions}. Conditions (i) and (ii) therein are evident. Condition (iii) follows from previous remark combined with $\int_{(0,1]}z^2 \mu_{\alpha(x)}(dz) = \frac{f(x)}{2-\alpha(x)}$ and $\int_{(1,\infty)}z \mu_{\alpha(x)}(dz) = \frac{f(x)}{\alpha(x)-1}$. Hence conditions (A1) -- (A4) are satisfied. Condition (A5) is trivial since $\rho \equiv 1$ while (A6) follows from
    \[
     \int_{\X \backslash \{0\}}\sum_{x \in V}v(x)\nu(x) H_2(d\nu) = \sum_{w \in V}v(w)\int_{(0,\infty)}z \sigma_w(dz) < \infty.
    \]
    Assertions (a) and (b) are left for the reader.
\end{proof}

Assumptions \eqref{eq: alpha 1} and \eqref{eq: alpha 2} are natural to guarantee that $(J_t(x))_{t \geq 0}$ and the stochastic integrals against $(B_t(x))_{t \geq 0}$ and $(Z_t(x))_{t \geq 0}$ take values in $\X$. For constant $g$, \eqref{eq: alpha 2} implies that $\alpha(x)$ is bounded away from $1$. If $v \equiv 1$, then \eqref{eq: alpha 2} also implies that $\alpha(x)$ needs to be bounded away from $2$. However, in general, $\alpha$ may approach $1$ and $2$ provided that the singularities in the denominator are compensated by $g(x)$ and $v(x)$.

Let us close this presentation with two additional discrete examples previously studied in the literature.

\begin{Example}[nearest-neighbor continuous-state branching process with unit jumps]\label{Example nearest neighbors}
    Consider $V = \Z^d$, $c(x, t) \equiv c_0t$, $g(x, t) \equiv g_0t$, and $\rho \equiv 0$, where $c_0, g_0 \geq 0$. Let 
    \[
     B(x,\eta) =  \sum_{y \in \Z ^d:|y-x|=1} \eta(y),
    \]
    $H_1(x, \cdot) = \sum_{y \in \Z^d : |y-x| = 1 } \delta_{\delta_y}$, 
    and $v(x) = e^{-|x|_1}$, where $|x|_1$ is the $\ell^1$ norm of $x \in \Z ^d$.
    This gives a nearest-neighbor continuous-space branching process with unit jumps. It is straightforward to check that (A1)--(A6) are satisfied
    and Theorems \ref{thm existence uniqueness comparison}
    and \ref{thm: at most linear} hold.    
\end{Example}

\begin{Example}[branching random walk]
 Take $V = \Z ^d$ and $v(x) = e^{-|x|_1}$.
 Assuming that for every $x \in V$ the measure $H_1(x,\cdot)$ is concentrated on integer-valued elements of $\X$ and taking $ c \equiv 0$, $\rho \equiv 0$, $g(x, s) \equiv s$, and 
 \[
 B(x, \eta) = \eta (x) \int_{ \{\| \nu \| \geq 1 \}} \nu(x) H_1(x,d\nu).
 \]
 Suppose further that $H_1$ satisfies
 \[
  \sup\limits _{x \in \Z ^d} \int_{ \{\| \nu \| \geq 1 \}} \nu(x) H_1(x,d\nu) + \sup_{x \in \Z^d}e^{-|x|_1}\int_{\X \backslash \{0\}} \sum_{y \in V \backslash \{x\}} e^{-|y|_1}\nu(y) H_1(x,d\nu) =: C < \infty.
 \]
 Then conditions (A1) -- (A6) are satisfied with $C_1 = C_4 = C_6 = C$, and $C_2 = C_3 = C_5 = 0$, and we obtain a continuous-time discrete-space branching random walk (see e.g. \cite{BZ15, BMY21}). The process is inhomogeneous in space and can fit in the framework of  a branching random walk in a random environment (see e.g. \cite{BRWRE}) if the measures $H_1(x, \cdot)$ are additionally randomized. Finally, if
 \[
  \int_{\X \backslash \{0\}}\nu(x) H_1(y,d\nu) = 0, \qquad |x-y|_1 > R
 \]
 holds for some $R > 0$, then Theorem \ref{thm: at most linear} is applicable.
\end{Example}

\section{Non-explosion and first-moment estimate}

A solution $\eta$ of \eqref{SDE} with lifetime $\zeta$ consists of a stopping time $\zeta$ and a process $(\eta_t)_{t \in [0,\zeta)}$ such that \eqref{SDE} is satisfied on $\{t \leq \tau_m(\eta)\}$ for each $m \geq 1$, where 
\begin{align}\label{eq: taum}
 \tau_m(\eta) = \inf\left\{ t \in [0,\zeta) \ | \ \| \eta_t \| > m \right\}.
\end{align}
with the convention $\inf \emptyset = \infty$. Clearly, $\tau_m(\eta)$ is an increasing sequence of stopping times. Let us first prove that each solution of \eqref{SDE} is always conservative.
\begin{Theorem}\label{thm: conservative}
 Suppose that (A1) -- (A4), and (A6) are satisfied. Let $(\eta_t)_{t \in [0,\zeta)}$ be a solution of \eqref{SDE} with lifetime $\zeta$. 
 Let $\tau_m = \tau_m(\eta)$ be the stopping time defined in \eqref{eq: taum}. Then $\tau_m \nearrow \infty$ a.s. as $m \to \infty$.
\end{Theorem}
\begin{proof}
    The definition of $\tau_m$ implies that $\tau_m \leq \tau_{m+1}$ holds for each $m \geq 1$. Define $\sup_{m \geq 1}\tau_m = \tau$. Then we have to prove that $\P[\tau = \infty] = 1$. Fix $T > 0$, then $\P[\tau \leq T] = \lim_{m \to \infty} \P[\tau_m \leq T]$ and hence it suffices to prove that $\P[\tau_m \leq T] \to 0$ as $m \to \infty$. For this purpose, we note that 
    \begin{align}
        \P[ \tau_m \leq T ] \notag
        &= \P\left[ \sup_{t \in [0,T]}\|\eta_t\| > m\right]
        \\ \notag
        &\leq \frac{1}{m}\E\left[ \sup_{t \in [0,T]}\|\eta_t\|\right ]
        \\ &\leq \frac{1}{m}\sum_{x \in V}v(x) \E\left[ \sup_{t \in [0,T]}\eta_t(x) \right]. \label{cull}
    \end{align}
    Let $x \in V$ and let $\mathcal{M}_t(x)$ be the local martingale defined by
\begin{align*}
    \mathcal{M}_t(x) &= \int_0^t \sqrt{2c(x,\eta_s(x))}dW_s(x)
    + \int_0^t \int_{\X \backslash \{0\}}\int_{\R_+} \nu(x) \mathbbm{1}_{ \{ u \leq g(x,\eta_{s-}(x))\} } \widetilde{N}_x(ds,d\nu,du).
\end{align*}
 Using the Burkholder-Davis-Gundy inequality combined with (A2) and (A3), we find that 
\begin{align*}
    &\ \E\left[ \sup_{t \in [0,T \wedge \tau_m]} |\mathcal{M}_t(x)|\right]
    \\ 
    &\leq \left( \E\left[ \sup_{t \in [0,T \wedge \tau_m]}\left| \int_0^t \sqrt{2c(x,\eta_s(x))}dW_s(x) \right|^2 \right] \right)^{1/2}  \notag
    \\
    &\ \ \ + \left( \E\left[ \sup_{t \in [0,T \wedge \tau_m]}\left|\int_0^t \int_{\{\|\nu\| \leq 1\} \backslash \{0\}}\int_{\R_+} \nu(x) \mathbbm{1}_{ \{ u \leq g(x,\eta_{s-}(x))\} } \widetilde{N}_x(ds,d\nu,du) \right|^2 \right] \right)^{1/2} \notag
    \\
    &\ \ \ + \E\left[ \sup_{t \in [0,T \wedge \tau_m]}\left|\int_0^t \int_{\{\|\nu\| > 1 \}}\int_{\R_+} \nu(x) \mathbbm{1}_{ \{ u \leq g(x,\eta_{s-}(x))\} } \widetilde{N}_x(ds,d\nu,du) \right| \right]  \notag
    \\
    &\leq \sqrt{8}\left( \E\left[\int_0^{T} \1_{[0,\tau_m]}(s)c(x,\eta_s(x))ds \right] \right)^{1/2} \notag
    \\
    &\ \ \ + 2 \left( \E\left[ \int_0^{T} \int_{\{\|\nu\| \leq 1\} \backslash \{0\}}\int_{\R_+} \1_{[0,\tau_m]}(s)\nu(x)^2 \mathbbm{1}_{ \{ u \leq g(x,\eta_{s-}(x))\} } \widehat{N}_x(ds,d\nu,du)\right]\right)^{1/2} \notag
    \\
    &\ \ \ + 2 \E\left[ \int_0^{T} \int_{\{\|\nu\| > 1 \}}\int_{\R_+} \1_{[0,\tau_m]}(s)\nu(x) \mathbbm{1}_{ \{ u \leq g(x,\eta_{s-}(x))\} } \widehat{N}_x(ds,d\nu,du)  \right]  \notag
    \\
    &\leq \sqrt{8}\sqrt{C_2(x)}\left( \int_0^T \E\left[ \sup_{r \in [0,s \wedge \tau_m]}\eta_r(x) \right] ds \right)^{1/2} \notag
    \\
    &\ \ \ + 2 \sqrt{C_3(x)} \left( \int_{\{\|\nu\| \leq 1\} \backslash \{0\}} \nu(x)^2 H_1(x,d\nu) \right)^{1/2}\left( \int_0^{T} \E\left[ \sup_{r \in [0,s\wedge \tau_m]}\eta_r(x) \right] ds \right)^{1/2} \notag
    \\
    &\ \ \ + 2 C_3(x) \left( \int_{\{\|\nu\| > 1 \}}\nu(x) H_1(x,d\nu) \right) \int_0^T \E\left[ \sup_{r \in [0,s \wedge \tau_m]}\eta_r(x) \right] ds  \notag
\end{align*}
 Define $f_m(t;x) = \E\left[ \sup_{s \in [0,t \wedge \tau_m]} \eta_s(x)\right]$. Multiplying by $v(x)$ and summing over $x \in V$ we obtain
\begin{align*}
  &\ \sum_{x \in V}v(x)\E\left[ \sup_{t \in [0,T\wedge \tau_m]}|\mathcal{M}_t(x)| \right]
    \leq \sqrt{8} \sum_{x \in V}v(x)\sqrt{C_2(x)}\left( \int_0^T f_m(s;x) ds \right)^{1/2} 
    \\ &\ \ \ + 2 \sum_{x \in V}v(x)\sqrt{C_3(x)} \left( \int_{\{\|\nu\| \leq 1\} \backslash \{0\}} \nu(x)^2 H_1(x,d\nu) \right)^{1/2}\left( \int_0^{T} f_m(s;x) ds \right)^{1/2} 
    \\ &\ \ \ + 2 \sum_{x \in V} v(x) C_3(x) \left( \int_{\{\|\nu\| > 1 \}}\nu(x) H_1(x,d\nu) \right) \int_0^T f_m(s;x)ds 
    \\ &\leq \sqrt{8} \left(\sum_{y \in V}v(y)C_2(y) \right) \int_0^T \sum_{x \in V}v(x)f_m(s;x) ds  
    \\ &\ \ \ + 2 \left(\sum_{y \in V}v(y)C_3(y) \int_{\{\|\nu\| \leq 1\} \backslash \{0\}} \nu(y)^2 H_1(y,d\nu) \right) \int_0^{T} \sum_{x \in V}v(x)f_m(s;x) ds 
    \\ &\ \ \ + 2 C_4 \int_0^T \sum_{x \in V}v(x)f_m(s;x)ds 
    \\ &=: c_0 \int_0^T \sum_{x \in V}v(x)f_m(s;x)ds
\end{align*}
with a constant $c_0 \in (0,\infty)$. Using \eqref{SDE} combined with (A6), we find that
\begin{align*}
    &\ \E\left[ \sup_{t \in [0,T \wedge \tau_m]}\eta_t(x)\right] 
    \\ &\leq \E[\eta_0(x)] + \E\left[ \sup_{t \in [0,T\wedge \tau_m]}|\mathcal{M}_t(x)| \right] + \E\left[ \int_0^{T} \1_{[0,\tau_m]}(s)|B_0(x,\eta_{s})|ds\right] 
    \\ &\ \ \ + \sum_{y \in V \backslash \{x\}} \E\left[\int_0^{T} \int_{\X \backslash \{0\}}\int_{\R_+} \1_{[0,\tau_m]}(s)\nu(x) \mathbbm{1}_{ \{ u \leq g(y,\eta_{s-}(y))\} } N_y(ds,d\nu,du) \right] 
    \\ &\ \ \ + \E\left[\int_{0}^{T} \int_{\X \backslash \{0\}} \int_{\R_+}\1_{[0,\tau_m]}(s)\nu(x) \mathbbm{1}_{ \{ u \leq \rho(x,\eta_{s-},\nu)\} } M(ds,d\nu,du) \right]
    \\ &\leq \E[\eta_0(x)] + \E\left[ \sup_{t \in [0,T\wedge \tau_m]}|\mathcal{M}_t(x)| \right] + \int_0^{T} \E\left[\1_{[0,\tau_m]}(s)|B_0(x,\eta_{s})|\right] ds
    \\ &\ \ \ + \sum_{y \in V \backslash \{x\}} C_3(y)\left( \int_{\X \backslash \{0\}} \nu(x) H_1(y,d\nu) \right)\int_0^T f_m(s;y) ds
    \\ &\ \ \ + \int_0^T \int_{\X \backslash \{0\}} \nu(x) \E\left[ \1_{[0,\tau_m]}(s) \rho(x,\eta_{s}, \nu) \right] H_2(d\nu) ds.
\end{align*}
After multiplying the last inequality by $v(x)$ and taking the sum over $x$.
We are going to estimate the last three terms separately.
Recall that $\|B_0(\cdot,\eta)\|=\sum_{x\in V}v(x)B_0(x,\eta)$.
For the first one we obtain by (A6)
\begin{align*}
    \sum_{x \in V}v(x)\int_0^{T} \E\left[\1_{[0,\tau_m]}(s)|B_0(x,\eta_{s})|\right] ds
    &= \int_0^{T} \E\left[\1_{[0,\tau_m]}(s)\|B_0(\cdot,\eta_{s})\|\right] ds
    \\ &\leq C_6T + C_6\int_0^T \sum_{x \in V}v(x)f_m(s;x)ds.
\end{align*}
The second term gives by (A4)
\begin{align*}
   &\ \sum_{x \in V}v(x)\sum_{y \in V \backslash \{x\}} C_3(y)\left( \int_{\X \backslash \{0\}} \nu(x) H_1(y,d\nu) \right)\int_0^T f_m(s;y) ds
   \\ &= \sum_{y \in V}C_3(y)\left( \int_{\X \backslash \{0\}} \sum_{x \in V \backslash \{y\}} v(x)\nu(x) H_1(y,d\nu) \right)\int_0^T f_m(s;y) ds
   \\ &\leq C_4 \int_0^T \sum_{y \in V}v(y)f_m(s;y)ds.
\end{align*}
Finally, for the last term, we obtain by (A6)
\begin{align*}
    &\ \sum_{x \in V}v(x)\int_0^T \int_{\X \backslash \{0\}} \nu(x) \E\left[ \1_{[0,\tau_m]}(s) \rho(x,\eta_{s}, \nu) \right] H_2(d\nu) ds
    \\ &= \int_0^T \E\left[ \1_{[0,\tau_m]}(s) \int_{\X \backslash \{0\}} \sum_{x \in V}v(x)\nu(x) \rho(x,\eta_{s}, \nu) H_2(d\nu)  \right] ds
    \\ &\leq C_6T +  C_6\int_0^T \sum_{x \in V}v(x)f_m(s;x) ds.
\end{align*}
 Together these inequalities yield
 \begin{align*}
     \sum_{x \in V}v(x) f_m(T;x) 
     &\leq \E[\|\eta_0\|] + 2C_6T + \left(c_0 + 2C_6 + C_4\right)\int_0^T \sum_{x \in V}v(x)f_m(s;x)ds.
 \end{align*}
 The Gronwall inequality yields
 \[
  \sum_{x \in V}v(x) f_m(T;x) \leq \left(\E[\|\eta_0\|] + 2C_6T\right)e^{\left(c_0 + 2C_6 + C_4\right)T}.
 \]
 Letting $m \to \infty$ and using Fatou's lemma gives
 \begin{align*}
     \sum_{x \in V}v(x) \E\left[ \sup_{t \in [0,T]}\eta_t(x) \right]
     &\leq \sup_{m \geq 1} \sum_{x \in V}v(x) f_m(T;x) 
     \\ &\leq \left(\E[\|\eta_0\|] + 2C_6T\right)e^{\left(c_0 + 2C_6 + C_4\right)T} < \infty.
 \end{align*}
 With this, the statement of the theorem follows from \eqref{cull}.
\end{proof}

Next we prove a simple but useful observation used for the localization of coefficients.
\begin{Lemma}\label{lemma: localization}
 Let $(\Omega, \F, (\F_t)_{t \geq 0}, \P)$ be a stochastic basis with the usual conditions and let $(\eta_t)_{t \geq 0}$ be an $(\F_t)_{t \geq 0}$-adapted cadlag process such that $\|\eta_t\| < \infty$ holds a.s. Define $\tau_m(\eta)$ as in \eqref{eq: taum} with $\zeta = + \infty$. Then $(\tau_m(\eta))_{m \in \N}$ is an increasing sequence of stopping times satisfying $\tau_m(\eta) \nearrow \infty$ a.s. as $m \to \infty$.
 Finally, one has $\| \eta_{t-}\| \leq m$ and $\eta_{t-}(x) \leq \frac{m}{v(x)}$ a.s. for each $t \in [0, \tau_m]$ and $x \in V$.
\end{Lemma}
\begin{proof}
 Since $(\F_t)_{t \geq 0}$ is right-continuous and $(\eta_t)_{t \geq 0}$ has cadlag paths, it follows that $\tau_m(\eta)$ is an $(\F_t)_{t \geq 0}$-stopping time.
 Next observe that, by definition, $\tau_m(\eta) \leq \tau_{m+1}(\eta)$ holds a.s. for each $m \in \N$. Let
 $\tau(\eta) := \sup_{m \in \N} \tau_m(\eta)$.
 Then for all $t > 0$ we obtain 
 \begin{align*}
   \P[ \tau(\eta) > t] &= \lim_{m \to \infty} \P[ \tau_m(\eta) > t]
   \\ &= \lim_{m \to \infty} \P[ \| \eta_{t} \| \leq m ] 
   = \P[ \| \eta_{t} \| < \infty ] =  1.
 \end{align*} 
 Letting $t \to \infty$ yields $\tau(\eta) = \infty$ a.s.. 
 The property $\| \eta_{t-}\| \leq m$ for $t \in [0,\tau_m]$ holds by definition of $\tau_m$, while the second inequality follows from
 \begin{displaymath}
  \eta_{t-}(x) \leq \frac{1}{v(x)} \| \eta_{t-}\| \leq \frac{m}{v(x)}.
 \end{displaymath}
\end{proof}
Finally, we prove a first-moment estimate for the solutions of \eqref{SDE} with parameters depending locally uniformly on the constants appearing in (A1) -- (A6).
\begin{Theorem}\label{first_moment_estimate}
 Suppose that (A1) -- (A4), and (A6) are satisfied. Then there exists a constant $C > 0$ such that each weak solution $(\eta_t)_{t \geq 0}$ of \eqref{SDE} satisfies 
 \begin{align}\label{unverkennbar = unmistakable}
     \E[ \| \eta_t \| ] \leq \left(1 + \E[ \| \eta_0 \| ]\right) e^{Ct}, \qquad t \geq 0.
 \end{align}
 Furthermore, let $(\xi_t)_{t \geq 0}$ be a weak solution to \eqref{SDE} with different functions $\tilde B, \tilde B_0, \tilde B_1, \tilde c, \tilde g, 
 \tilde \rho$ instead of $ B, B_0, B_1,  c,  g, \rho$, respectively. Assume that (A1) -- (A4) and (A6) are satisfied for $\tilde B, \tilde B_0, \tilde B_1, \tilde c, \tilde g, \tilde \rho$, and for all $\alpha, \beta \in \X$, 
$\alpha \leq \beta$, $x \in V$, and $0 \leq s \leq t$ we have
 $\tilde B_0(x, \alpha) \leq B_0(x, \beta)$, $\tilde c(x,s ) \leq c(x,t)$, $\tilde g(x,s) \leq g(x,t)$, $\tilde \rho(x, \alpha, \nu) \leq \rho(x, \beta, \nu)$. Then
 \begin{align}\label{bieder = tame}
     \E[ \| \xi_t \| ] \leq \left(1 + \E[ \| \xi_0 \| ]\right) e^{Ct}, \qquad t \geq 0.
 \end{align}
 with the same constant $C$ as in \eqref{unverkennbar = unmistakable}.
\end{Theorem}
\begin{proof}
 The main part of the proof is establishing \eqref{unverkennbar = unmistakable}. Let $\tau_m(\eta)$ be the stopping time defined in Lemma \ref{lemma: localization}. Observe that
  \begin{align*}
     \eta_t(x) &= \eta_0(x) + \int_0^t \left( B(x,\eta_s)
     + \sum_{y \in V \backslash \{x\}} \int_{\X \backslash \{0\}} \int_{\R_+} \nu(x) \1_{ \{ u \leq g(y,\eta_{s-}(y))\} } H_1(y,d\nu)du \right)ds
 \\ &\ \ \ + \int_{0}^t \int_{\X \backslash \{0\}} \int_{\R_+}\nu(x) \mathbbm{1}_{ \{ u \leq \rho(x, \eta_{s-},\nu)\} } M(ds,d\nu,du) + \mathcal{M}_t(x)
 \end{align*}
 with $(\mathcal{M}_t(x))_{t \geq 0}$ given by
 \begin{align*}
  \mathcal{M}_{t}(x) &:= \int_0^t \sqrt{2c(x,\eta_s(x))}dW_s(x)
  \\ &\ \ \ + \int_0^t \int_{ \{ \| \nu \| \leq 1\} \backslash \{0\}}\int_{\R_+} \nu(x) \mathbbm{1}_{ \{ u \leq g(x,\eta_{s-}(x))\} } \widetilde{N}_x(ds,d\nu,du)
  \\ &\ \ \ + \int_0^t \int_{ \{ \| \nu \| > 1 \}}\int_{\R_+} \nu(x) \mathbbm{1}_{ \{ u \leq g(x,\eta_{s-}(x))\} } \widetilde{N}_x(ds,d\nu,du)
  \\ &\ \ \ + \sum_{y \in V \backslash \{x\}}\int_0^t \int_{\X \backslash \{0\}}\int_{\R_+} \nu(x) \mathbbm{1}_{ \{ u \leq g(y,\eta_{s-}(y))\} } \widetilde{N}_y(ds,d\nu,du).
 \end{align*}
 Next we prove that $(\mathcal{M}_{t \wedge \tau_m}(x))_{t\geq 0}$ is a martingale for each $m \geq 1$ and $x \in V$. Indeed, the first two terms are square integrable martingales, since by (A2) we have
 \begin{align}
 \notag
     \E\left[ \left| \int_0^{t \wedge \tau_m} \sqrt{2 c(x,\eta_s(x))} dW_s(x)\right|^2 \right ]
     &= \E \left[ \int_0^{t \wedge \tau_m} 2 c(x,\eta_s(x)) ds \right]
     \\ &\leq 2C_2(x)\frac{m}{v(x)} t < \infty,
     \label{der Teich = pond2}
 \end{align}
 and by (A3) and (A4) also
 \begin{align*}
     &\ \E\left[ \int_0^{t \wedge \tau_m} \int_{ \{ \| \nu \| \leq 1\} \backslash \{0\}}\int_{\R_+} \left|\nu(x) \mathbbm{1}_{ \{ u \leq g(x,\eta_{s-}(x))\} } \right|^2 duH_1(x,d\nu)ds\right]
     \\
     &= \E\left[ \int_0^{t \wedge \tau_m} \int_{ \{ \| \nu \| \leq 1\} \backslash \{0\}}\nu(x)^2 g(x,\eta_{s-}(x))  H_1(x,d\nu)ds \right]
     \\
     &\leq C_3(x) t \int_{ \{ \| \nu \| \leq 1\} \backslash \{0\}}\nu(x)^2 H_1(x,d\nu) < \infty.
 \end{align*}
 The third term is a martingale due to (A4) and
 \begin{align*}
     &\ \E\left[ \int_0^{t \wedge \tau_m} \int_{ \{ \| \nu \| > 1\}}\int_{\R_+} \left|\nu(x) \mathbbm{1}_{ \{ u \leq g(x,\eta_{s-}(x))\} } \right|  du H_1(x,d\nu)ds\right]
     \\
     &\leq  C_3(x)t \int_{ \{ \| \nu \| > 1\}} \nu(x) H_1(x,d\nu) < \infty.
 \end{align*}
 Finally, the last term is a martingale since by (A3)
 \begin{align*}
  &\ \sum_{y \in V \setminus \{x\}} \E\left[ \left| \int_0^{t \wedge \tau_m} \int_{\X \backslash \{0\}}\int_{\R_+} \nu(x) \mathbbm{1}_{ \{ u \leq g(y,\eta_{s-}(y))\} } \widetilde{N}_y(ds,d\nu,du) \right| \right]
  \\
  &\leq 2 \sum_{y \in V \setminus \{x\} } \E\left[ \int_{0}^{t \wedge \tau_m} \int_{\X \backslash \{0\} } \nu(x) g(y,\eta_{s-}(y)) H_1(y,d\nu) ds\right]
  \\
  &\leq \frac{2 }{v(x)} \sum_{y\in V \setminus \{x\} }C_3(y) \E\left[ \int_0^{t \wedge \tau_m}\eta_s(y)ds \right] \int_{\X \setminus \{0\} } \nu(x)v(x) H_1(y,d\nu) 
  \\
  &\leq \frac{2}{v(x)}\sum_{y \in V \setminus \{x\}} C_3(y) \E\left[ \int_0^{t \wedge \tau_m}\eta_s(y)ds \right] \int_{\X \setminus \{0\} } \sum_{w \in V \setminus \{y\}}\nu(w)v(w) H_1(y,d\nu) 
  \\
  &\leq \frac{C_4}{v(x)}\sum_{y \in V \setminus \{x\}} \E\left[ \int_0^{t \wedge \tau_m}\eta_s(y)ds \right]v(y)
  \\
  &\leq \frac{C_4}{v(x)}\E\left[ \int_0^{t \wedge \tau_m}\|\eta_s\| ds \right]
  \\
  &\leq \frac{C_4 m}{v(x)}t < \infty.
 \end{align*}
 This proves that $(\mathcal{M}_{t \wedge \tau_m}(x))_{t \geq 0}$ is a martingale.
 Hence taking expectations and using optimal stopping for integrable martingales, gives
 \begin{align*}
     \E[ \eta_{t \wedge \tau_m}(x) ] &= \E[ \eta_0(x) ] + \E\left[ \int_0^{t \wedge \tau_m} ( B_0(x,\eta_s) - B_1(x,\eta_s(x))) ds \right]
     \\ &\ \ \ + \E\left[\sum_{y \in V \backslash \{x\}} \int_{\X \backslash \{0\})} \nu(x) g(y,\eta_{s-}(y)) H_1(y,d\nu) \right]
     \\ &\ \ \ + \E\left[ \int_0^{t \wedge \tau_m} \int_{\X \backslash \{0\}} \nu(x) \rho(x,\eta_{s-}, \nu) H_2(d\nu)ds \right]
     \\ &\leq \E[ \eta_0(x) ] + \E\left[ \int_0^{t \wedge \tau_m} B_0(x,\eta_s)  ds \right]
     \\ &\ \ \ + \E\left[\sum_{y \in V \backslash \{x\}} \int_{\X \backslash \{0\}} \nu(x) g(y,\eta_{s-}(y)) H_1(y,d\nu) \right]
     \\ &\ \ \ + \E\left[ \int_0^{t \wedge \tau_m} \int_{\X \backslash \{0\}} \nu(x) \rho(x,\eta_{s-}, \nu) H_2(d\nu)ds \right]
 \end{align*}
 where we have used that $B_1$ is non-decreasing so that $B_1(x,\eta(x)) \geq B_1(x,0) = 0$. This yields by (A3), (A4), and (A6)
\begin{align*}
   \E\left[ \| \eta_{t\wedge \tau_m} \|  \right]
   &= \sum_{x \in V}v(x) \E[  \eta_{t \wedge \tau_m}(x) ]
     \\ &\leq \E[ \| \eta_0 \| ] + \E\left[ \int_0^{t \wedge \tau_m} \sum_{x \in V} v(x)B_0(x, \eta_s) ds \right]
     \\ &\ \ \ + \E\left[\sum_{x \in V}\sum_{y \in V \backslash \{x\}} \int_{\X \backslash \{0\}} v(x)\nu(x) g(y,\eta_{s-}(y)) H_1(y,d\nu) \right]
     \\ &\ \ \ + \E \left[ \int_0^{t \wedge \tau_m} \int_{\X \backslash \{0\} } \sum_{x \in V}v(x)\nu(x) \rho(x,\eta_{s-},\nu) H_2(d\nu) ds \right]
     \\ &\leq \E[ \| \eta_0 \| ] + 2C_6 t + (2C_6 + C_4) \int_0^t \E[ \| \eta_{s \wedge \tau_m}\|]ds
 \end{align*}
 where we have used (A3) and (A4) so that
 \begin{align*}
  &\ \sum_{x \in V} \sum_{y \in V \backslash \{x\}} \int_{\X \backslash \{0\}} v(x)\nu(x) g(y,\eta_{s-}(y)) H_1(y,d\nu)
  \\ &= \sum_{y \in V} g(y, \eta_{s-}(y))   \int_{\X \backslash \{0\}} \sum_{x \in V \backslash \{y\}} \nu(x)v(x) H_1(y,d\nu)
  \\ &\leq \sum_{y \in V} \frac{C_4 v(y)}{C_3(y)} g(y, \eta_{s-}(y)) 
  \\ &\leq C_4 \| \eta_{s-}\|.
 \end{align*}
  Inequality \eqref{unverkennbar = unmistakable} now follows from the Gronwall lemma.
  The proof of \eqref{bieder = tame} follows exactly the same path, we just need to replace $\tilde B, \tilde B_0, \tilde B_1, \tilde c, \tilde g, 
  \tilde \rho$ with $ B,  B_0,  B_1,  c, g, \rho$ respectively along the way. 
   For example, instead of \eqref{der Teich = pond2} we write
   \begin{align*}
     \E\left[ \left| \int_0^{t \wedge \tau_m} \sqrt{2 \tilde c(x,\xi_s(x))} dW_s(x)\right|^2 \right ]
     &= \E \left[ \int_0^{t \wedge \tau_m} 2 \tilde  c(x,\xi_s(x)) ds \right]
     \\
     &\leq \E \left[ \int_0^{t \wedge \tau_m} 2 c(x,\xi_s(x)) ds \right]
     \\
     &\leq 2C_2(x)\frac{m}{v(x)} t < \infty.
 \end{align*}
\end{proof}

\section{Pathwise uniqueness}

In this section we prove the pathwise uniqueness of the solution under slightly weaker conditions, i.e., we consider:
\begin{enumerate}
   \item[(A1')] The drift coefficient $B(x,\eta)$ has the form 
   $B(x,\eta) = B_0(x,\eta) - B_1(x,\eta(x))$ where $B_0(x,\cdot): \X \longrightarrow \R_+$ and $B_1(x,\cdot): \R_+ \longrightarrow \R_+$ are measurable mappings for each $x \in V$. Moreover, for each $R > 0$ there exists a constant $C_1(R) > 0$ such that
   \[
   \|B_0(\cdot, \eta) - B_0(\cdot, \xi) \| \leq C_1(R) \| \eta - \xi \|, \qquad x \in V,
   \]
   holds for all $\eta,\xi \in \X$ with $\| \eta\|, \| \xi \| \leq R$.
   Finally, the function $\R_+ \ni t \longmapsto B_1(x,t)$ is continuous and non-decreasing satisfying $B_1(x,0) = 0$ for each $x \in V$.
   
    \item[(A5')] For each $R > 0$ there exists a constant $C_5(R) > 0$ such that
    \[
     \int_{\X \backslash \{0\}} \sum_{x \in V}v(x)\nu(x) | \rho(x,\eta, \nu) - \rho(x,\xi, \nu)| H_2(d\nu) \leq C_5(R)\| \eta - \xi\|
    \]
    holds for all $\eta,\xi \in \X$ with $\| \eta \|, \|\xi\| \leq R$.
\end{enumerate}
In contrast to (A1) and (A5), the above conditions do not require that $B$ and $\rho$ are monotone with respect to the configuration $\eta$. 
The following is our main result on the uniqueness of \eqref{SDE}.
\begin{Theorem}\label{thm: uniqueness}
 Let $(\eta_t)_{t \geq 0}$ and $(\xi_t)_{t \geq 0}$ be two weak solutions to \eqref{SDE} defined on the same stochastic basis $(\Omega, \F, (\F_t)_{t \geq 0},\P)$ and suppose that conditions (A1'), (A2) -- (A4), and (A5') are satisfied. Let $\tau_m(\eta), \tau_m(\xi)$ be the stopping times defined in Lemma \ref{lemma: localization} and set $\tau_m := \tau_m(\eta) \wedge \tau_m(\xi)$. Then
 \[
   \E[ \| \eta_{t \wedge \tau_m} - \xi_{t \wedge \tau_m} \| ]
   \leq \E[ \| \eta_0 - \xi_0 \| ] e^{\left( C_1(m) + 2C_4 + C_5(m) \right)t}, \qquad t \geq 0
 \]
 holds for each $m \geq 1$.
 In particular, if $\eta_0 = \xi_0$ holds a.s., then 
 $\P[ \eta_t = \xi_t, \ t \geq 0] = 1$, i.e. pathwise uniqueness among weak solutions to \eqref{SDE} holds.
\end{Theorem}
\begin{proof}
  Define $\zeta_t := \eta_t - \xi_t$ and fix $x \in V$. Then
 \begin{align*}
  \zeta_t(x) 
  &= \zeta_0(x) + \int_0^t \left(B(x,\eta_s) - B(x,\xi_s) \right)ds 
    \\ &\ \ \ + \int_0^t \left( \sqrt{2c(x,\eta_s(x))} - \sqrt{2 c(x,\xi_s(x))}\right)dW_s(x)
    \\ &\ \ \ + \int_0^t \int_{\X \backslash \{0\}}\int_{\R_+} \nu(x) \left(\mathbbm{1}_{ \{ u \leq g(x,\eta_{s-}(x))\} } - \mathbbm{1}_{ \{ u \leq g(x,\xi_{s-}(x))\} } \right) \widetilde{N}_x(ds,d\nu,du)
    \\ &\ \ \ + \sum_{y \in V \backslash \{x\}}\int_0^t \int_{\X \backslash \{0\}}\int_{\R_+} \nu(x) \left(\mathbbm{1}_{ \{ u \leq g(y,\eta_{s-}(y))\} } - \mathbbm{1}_{ \{ u \leq g(y,\xi_{s-}(y))\} } \right) N_y(ds,d\nu,du) 
    \\ &\ \ \ + \int_{0}^t \int_{\X \backslash \{0\}} \int_{\R_+}\nu(x) \left(\mathbbm{1}_{ \{ u \leq \rho(x,\eta_{s-},\nu)\} } - \mathbbm{1}_{ \{ u \leq \rho(x,\xi_{s-},\nu)\} } \right) M(ds,d\nu,du).
 \end{align*}
 Let $\phi_k: \R \longrightarrow \R_+$ be a sequence of twice continuously differentiable functions such that for each $k \geq 1$,
 \begin{enumerate} \label{phi k}
     \item[(i)] $\phi_k(-z) = \phi_k(z) \nearrow |z|$ as $k \to \infty$,
     \item[(ii)] $\phi_k'(z) \in [0,1]$ for $z \geq 0$ and $\phi_k'(z) \in [-1,0]$ for $z \leq 0$,
     \item[(iii)] $\phi_k''(z)|z| \leq 2/k$ holds for all $z \in \R$.
 \end{enumerate} 
 The construction of such a function follows the same arguments as the classical Yamada-Watanabe theorem for pathwise uniqueness. To simplify the notation below, we set $D_h\phi_k(z) := \phi_k(z+h) - \phi_k(z)$ for $z,h \in \R$. Let $z,h \in \R$ such that $zh \geq0$. Then using the mean-value theorem one can check that 
 \begin{align}\label{eq: estimate on phik}
  D_h\phi_k(z) \leq |h|, \ \ D_h\phi_k(z) - \phi_k'(z)h \leq \frac{h^2}{k|z|}, \ \  \text{ and }
  D_h\phi_k(z) - \phi_k'(z)h \leq |h|.  
 \end{align}
 Applying the It\^{o} formula to $\zeta_t(x)$ gives
 \begin{align}\label{eq:00}
 \phi_k(\zeta_t(x)) = \phi_k(\zeta_0(x)) 
     + \sum_{j=1}^5 \mathcal{R}_j(t) + \mathcal{M}(t),
 \end{align}
 where the processes $\mathcal{R}_1, \dotsc, \mathcal{R}_5$ are given by
 \label{mathcal R}
  \begin{align*}
 \mathcal{R}_1(t) &= \int_0^t \phi_k'(\zeta_s(x))\left( B(x,\eta_s) - B(x,\xi_s) \right)ds
 \\ \mathcal{R}_2(t) &= \frac{1}{2}\int_0^t \phi_k''(\zeta_s(x))\left( \sqrt{2c(x,\eta_s(x))} - \sqrt{2 c(x,\xi_s(x))}\right)^2 ds
 \\ \mathcal{R}_3(t) &= \int_0^t \int_{\X \backslash \{0\}} \int_{\R_+} \left( D_{\Delta_0(x,s)}\phi_k(\zeta_{s-}(x)) - \phi_k'(\zeta_{s-}(x))\Delta_0(x,s) \right) ds H_1(x,d\nu)du
  \\ \mathcal{R}_4(t) &= \sum_{y \in V \backslash \{x\}}\int_0^t \int_{\X \backslash \{0\}} \int_{\R_+} D_{\Delta_0(y,s)}\phi_k(\zeta_{s-}(x)) ds H_1(y,d\nu)du
 \\ \mathcal{R}_5(t) &= \int_0^t \int_{\X \backslash \{0\}}\int_{\R_+} D_{\Delta_1(x,s)}\phi_k(\zeta_{s-}(x)) ds H_2(d\nu)du
 \end{align*}
 with increments given by
 \begin{align*}
     \Delta_0(z,s) &= \nu(x) \left(\mathbbm{1}_{ \{ u \leq g(z,\eta_{s-}(z))\} } - \mathbbm{1}_{ \{ u \leq g(z,\xi_{s-}(z))\} }\right),
     \\ \Delta_1(z,s) &= \nu(z) \left(\mathbbm{1}_{ \{ u \leq \rho(z,\eta_{s-},\nu)\} } - \mathbbm{1}_{ \{ u \leq \rho(z,\xi_{s-},\nu)\} } \right).
 \end{align*}
 Note that $\Delta_0(z,s)$ and $\Delta_1(z,s)$ also depend on $u \geq 0$. The process $(\mathcal{M}(t))_{t \geq 0}$  given by
 \begin{align}\label{eq: martingale pathwise uniqueness}
     \mathcal{M}(t) &= \int_0^t \phi_k'(\zeta_s(x)) \left( \sqrt{2c(x,\eta_s(x))} - \sqrt{2 c(x,\xi_s(x))}\right)dW_s(x)
     \\ \notag &\ \ \ + \sum_{y \in V}\int_0^t \int_{\X \backslash \{0\}} \int_{\R_+}  D_{\Delta_0(y,s)}\phi_k(\zeta_{s-}(x)) \widetilde{N}_y(ds,d\nu,du)
     \\ \notag &\ \ \ + \int_0^t \int_{\X \backslash \{0\}}\int_{\R_+} D_{\Delta_1(x,s)}\phi_k(\zeta_{s-}(x)) \widetilde{M}(ds,d\nu,du),
 \end{align}
 is a local martingale.
 Here $\widetilde{M} = M - \widehat{M}$ denotes the compensated Poisson random measure. Note that $\mathcal{R}_j$ and $\mathcal{M}(t)$ do also depend on the previously fixed point $x$. Recall that $\tau_m$ satisfies, by 
 Lemma \ref{lemma: localization}, $\tau_m \longrightarrow \infty$ and it holds that
 \begin{align}\label{eq: boundedness by localization}
  \eta_{s-}(x), \xi_{s-}(x) \leq \frac{m}{v(x)} \ \ \text{ and }\ \ \| \eta_{s-} \|, \| \xi_{s-}\| \leq m, \qquad s \in [0, \tau_m], \ \ x \in V.
 \end{align}
 Using Property \eqref{eq: boundedness by localization} it is not difficult to see that $(\mathcal{M}(t \wedge \tau_m))_{t \geq 0}$ is a martingale for each $k$ and each $m$. For sake of completeness, the proof is given in the appendix. Below we will show that
  \begin{align}
  \notag 
   \mathcal{R}_1(t \wedge \tau_m) &\leq \int_0^{t \wedge \tau_m} |B_0(x,\eta_s) - B_0(x,\xi_s)|ds,
   \\  \mathcal{R}_2(t \wedge \tau_m) &\leq C_2(x) \frac{2t}{k}, \label{verwirrend}
   \\ \notag \mathcal{R}_3(t \wedge \tau_m) &\leq \frac{C_3(x)t}{k}\int_{ \{ \| \nu \| \leq 1\} \backslash \{0\}} \nu(x)^2 H_1(x,d\nu) 
   \\ \notag &\qquad + C_3(x)\left( \int_{\{\| \nu \| > 1\}} \nu(x)H_1(x,d\nu) \right) \int_0^{t \wedge \tau_m}|\eta_{s-}(x) - \xi_{s-}(x)| ds,
   \\ \mathcal{R}_4(t \wedge \tau_m) &\leq \sum_{y \in V \backslash \{x\}} \left( \int_{\X \backslash \{0\}} \nu(x) H_1(y,d\nu) \right) C_3(y) \int_0^{t \wedge \tau_m}|\eta_{s-}(y) - \xi_{s-}(y)| ds,
    \\ \notag \mathcal{R}_5(t \wedge \tau_m) &\leq \int_0^{t \wedge \tau_m} \int_{\X \backslash \{0\}} \nu(x) |\rho(x,\eta_{s-}, \nu) - \rho(x,\xi_{s-},\nu)| H_2(d\nu)ds.
  \end{align}
  Taking then expectations in \eqref{eq:00} and using the above estimates gives 
  \begin{align*}
      \E[ \phi_k(\zeta_{t \wedge \tau_m}(x))]
      &= \E[ \phi_k(\zeta_0(x)) ]
      + \E \left[ \sum_{j=1}^{5}\mathcal{R}_j(t \wedge \tau_m) \right]
      \\
      &\leq \E[ \phi_k(\zeta_0(x)) ]
      + \E\left[ \int_0^{t \wedge \tau_m} |B_0(x,\eta_s) - B_0(x,\xi_s)| ds \right]
      \\
      &\ \ \ + C_2(x) \frac{2t}{k}
      + \frac{C_3(x)t}{k}\int_{ \{ \| \nu \| \leq 1\} \backslash \{0\}} \nu(x)^2 H_1(x,d\nu) 
      \\
      &\ \ \ + C_3(x)\left( \int_{\{\| \nu \| > 1\}} \nu(x) H_1(x,d\nu) \right) \E\left[ \int_0^{t \wedge \tau_m} |\eta_{s-}(x) - \xi_{s-}(x)| ds \right]
      \\
      &\ \ \ + \sum_{y \in V \backslash \{x\}} \left( \int_{\X \backslash \{0\}} \nu(x) H_1(y,d\nu) \right) C_3(y)\E\left[ \int_0^{t \wedge \tau_m}|\eta_{s-}(y) - \xi_{s-}(y)| ds\right]
      \\
      &\ \ \ + \E\left[\int_0^{t \wedge \tau_m} \int_{\X \backslash \{0\}} \nu(x) |\rho(x,\eta_{s-}, \nu) - \rho(x,\xi_{s-},\nu)| H_2(d\nu)ds \right].
  \end{align*}
  Letting first $k \to \infty$, then multiplying each term by $v(x)$  and summing up over $x$, and finally using (i) yields 
  \begin{align*}
      &\ \E[ \| \eta_{t \wedge \tau_m} - \xi_{t \wedge \tau_m} \| ]
      \\ &= \sum_{x \in V}v(x) \E[ | \eta_{t \wedge \tau_m}(x) - \xi_{t \wedge \tau_m}(x)|]
      \\ &\leq \sum_{x \in V}v(x) \E[ |\eta_0(x) - \xi_0(x)| ]
      + \sum_{x \in V}v(x) \E\left[ \int_0^{t \wedge \tau_m}| B_0(x,\eta_{s-}) - B_0(x,\xi_{s-})| ds \right]
      \\ &\ \ \ + \sum_{x \in V}v(x)C_3(x)\left( \int_{\{\| \nu \| > 1\}} \nu(x) H_1(x,d\nu) \right) \E\left[ \int_0^{t \wedge \tau_m} |\eta_{s-}(x) - \xi_{s-}(x)| ds \right]
      \\ &\ \ \ + \sum_{x \in V}\sum_{y \in V \backslash \{x\}} \left( \int_{\X \backslash \{0\}} v(x)\nu(x) H_1(y,d\nu) \right) C_3(y) \E\left[\int_0^{t \wedge \tau_m}|\eta_{s-}(y) - \xi_{s-}(y)| ds \right]
      \\ &\ \ \ + \E\left[\int_0^{t \wedge \tau_m} \int_{\X \backslash \{0\}} \sum_{x \in V}v(x)\nu(x) |\rho(x,\eta_{s-}, \nu) - \rho(x,\xi_{s-},\nu)| H_2(d\nu)ds \right]
      \\ &\leq \E[ \| \eta_0 - \xi_0 \| ] + \E\left[ \int_0^{t \wedge \tau_m} \| B_0(\cdot, \eta_{s-}) - B_0(\cdot, \xi_{s-}) \| ds \right]
      \\ &\ \ \ + (2C_4 + C_5(m))\E\left[ \int_0^{t \wedge \tau_m} \| \eta_{s-} - \xi_{s-}\| ds \right]
      \\ &\leq \E[ \| \eta_0 - \xi_0 \| ] + (C_1(m) + 2C_4 + C_5(m)) \int_0^{t} \E\left[\| \eta_{s \wedge \tau_m} - \xi_{s \wedge \tau_m}\| \right] ds
  \end{align*}
  where we have used (A1'), (A4), (A5'), and
  \begin{align*}
   &\ \sum_{x \in V}\sum_{y \in V \backslash \{x\}} \left( \int_{\X \backslash \{0\}} v(x)\nu(x) H_1(y,d\nu) \right) C_3(y) \E\left[\int_0^{t \wedge \tau_m}|\eta_{s-}(y) - \xi_{s-}(y)| ds \right]
    \\ &= \sum_{y \in V}  \left( \int_{\X \backslash \{0\}} \sum_{x \in V \backslash \{y\}}v(x)\nu(x) H_1(y,d\nu) \right) C_3(y) \E\left[\int_0^{t \wedge \tau_m}|\eta_{s-}(y) - \xi_{s-}(y)| ds \right]
      \\ &\leq C_4\sum_{y \in V} v(y) \E\left[\int_0^{t \wedge \tau_m}|\eta_{s-}(y) - \xi_{s-}(y)| ds \right].
  \end{align*}
  The assertion of the theorem follows from the Gronwall lemma. Hence it remains to prove the estimates for $\mathcal{R}_j(t \wedge \tau_m)$, $j = 1,\dots, 5$ in \eqref{verwirrend}.
  For the first term, we obtain from (A1') combined with (ii) that 
  $\phi_k'(\zeta_s(x))(B_1(x,\eta_s(x)) - B_1(x,\xi_s(x))) \geq 0$
  holds a.s. for $s \in [0,t \wedge \tau_m]$.
  Hence we obtain 
  \begin{align*}
     \mathcal{R}_1(t \wedge \tau_m) 
      &= \int_0^{t \wedge \tau_m} \phi_k'(\zeta_s(x))(B_0(x,\eta_s) - B_0(x,\xi_s))ds 
      \\ &\ \ \ - \int_0^{t \wedge \tau_m} \phi_k'(\zeta_s(x))(B_1(x,\eta_s(x)) - B_1(x,\xi_s(x)))ds
      \\ &\leq \int_0^{t \wedge \tau_m} |B_0(x,\eta_s) - B_0(x,\xi_s)|ds
  \end{align*}
  For the second term, we first observe that (A2) and property (iii) yield
  \begin{align*}
      \phi_k''(\zeta_{s-}(x)) | c(x,\eta_{s-}(x)) - c(x,\xi_{s-}(x))|
     &\leq C_2(x)\phi_k''(\zeta_{s-}(x))| \eta_{s-}(x) - \xi_{s-}(x)|
     \\ &\leq C_2(x) \frac{2}{k},
  \end{align*}
  where we have used (A2). Hence using the elementary inequality $(a-b)^2 \leq |a^2 - b^2|$ for $a,b > 0$ for the first inequality, we find that
   \begin{align*}
      \mathcal{R}_2(t \wedge \tau_m) 
      &\leq \int_0^{t \wedge \tau_m} \phi_k''(\zeta_s(x)) |c(x,\eta_s(x)) - c(x,\xi_s(x))| ds
      \leq C_2(x) \frac{2t}{k}.
  \end{align*}
  To estimate the third term $\mathcal{R}_3$, we decompose the integral against $H_1(x,d\nu)$ into ${\{ \| \nu \| \leq 1\} \backslash \{0\}}$ and $\{ \| \nu \| > 1\}$ to find that $\mathcal{R}_3(t) = \mathcal{R}_3^1(t) + \mathcal{R}_3^2(t)$, where
  \begin{align*}
      \mathcal{R}_3^1(t) &= \int_0^t \int_{ \{ \| \nu \| \leq 1\} \backslash \{0\}} \int_{\R_+} \left(  D_{\Delta_0(x,s)}\phi_k(\zeta_{s-}(x)) - \phi_k'(\zeta_{s-}(x))\Delta_0(x,s) \right) ds H_1(x,d\nu)du
      \\
      \mathcal{R}_3^2(t) &= \int_0^t \int_{ \{ \| \nu \| > 1\}} \int_{\R_+} \left( D_{\Delta_0(x,s)}\phi_k(\zeta_{s-}(x)) - \phi_k'(\zeta_{s-}(x))\Delta_0(x,s) \right) ds H_1(x,d\nu)du.
  \end{align*}
  In order to estimate these integrals, we first compute the integral against $du$. Namely, observe that by (A3), $\zeta_{s-}(x) \leq 0$ implies that $g(x,\eta_{s-}(x)) \leq g(x,\xi_{s-}(x))$ and hence $\Delta_0(x,s) \leq 0$ while $\zeta_{s-}(x) > 0$ implies $\Delta_0(x,s) \geq 0$.
  Combining both observations we find for $\| \nu \| \leq 1$ by \eqref{eq: estimate on phik}
  \begin{align*}
     &\ \int_{\R_+} \left( D_{\Delta_0(x,s)}\phi_k(\zeta_{s-}(x)) - \phi_k'(\zeta_{s-}(x))\Delta_0(x,s) \right) du
     \\
     &= \int_{\R_+} \1_{ \{ \zeta_{s-}(x) > 0 \} }\left( D_{\Delta_0(x,s)}\phi_k(\zeta_{s-}(x)) - \phi_k'(\zeta_{s-}(x))\Delta_0(x,s) \right) du
     \\
     &\ \ \ +\int_{\R_+} \1_{ \{ \zeta_{s-}(x) \leq 0 \} }\left( D_{\Delta_0(x,s)}\phi_k(\zeta_{s-}(x)) - \phi_k'(\zeta_{s-}(x))\Delta_0(x,s) \right) du
     \\
     &= \int_{g(x,\xi_{s-}(x))}^{g(x,\eta_{s-}(x))} \1_{ \{ \zeta_{s-}(x) > 0 \} }\left( D_{\nu(x)}\phi_k(\zeta_{s-}(x)) - \phi_k'(\zeta_{s-}(x))\nu(x) \right) du
     \\
     &\ \ \ + \int_{g(x,\eta_{s-}(x))}^{g(x,\xi_{s-}(x))} \1_{ \{ \zeta_{s-}(x) \leq 0 \} }\left( D_{-\nu(x)}\phi_k(\zeta_{s-}(x)) + \phi_k'(\zeta_{s-}(x))\nu(x) \right) du
     \\
     &\leq \left| g(x,\eta_{s-}(x)) - g(x,\xi_{s-}(x))\right|\frac{\nu(x)^2}{k|\zeta_{s-}(x)|}
     \\
     &\leq \frac{C_3(x)}{k} \nu(x)^2
  \end{align*}
  while for $\| \nu \| > 1$, we obtain 
  \begin{align*}
      &\ \int_{\R_+} \left( D_{\Delta_0(x,s)}\phi_k(\zeta_{s-}(x)) - \phi_k'(\zeta_{s-}(x))\Delta_0(x,s) \right) du
     \\ &=  \int_{g(x,\xi_{s-}(x))}^{g(x,\eta_{s-}(x))} \1_{ \{ \zeta_{s-}(x) > 0 \} }\left( D_{\nu(x)}\phi_k( \zeta_{s-}(x)) - \phi_k'(\zeta_{s-}(x))\nu(x) \right) du
     \\ &\ \ \ + \int_{g(x,\eta_{s-}(x))}^{g(x,\xi_{s-}(x))} \1_{ \{ \zeta_{s-}(x) \leq 0 \} }\left( D_{-\nu(x)}\phi_k(\zeta_{s-}(x)) + \phi_k'(\zeta_{s-}(x))\nu(x) \right) du
     \\ &\leq \left| g(x,\eta_{s-}(x)) - g(x,\xi_{s-}(x))\right| \nu(x)
     \\ & \leq C_3(x)\nu(x)|\eta_{s-}(x) - \xi_{s-}(x)|.
  \end{align*}
 For the first part, we obtain 
  \begin{align*}
      \mathcal{R}_3^1(t \wedge \tau_m) 
      \leq \frac{C_3(x)t}{k} \int_{ \{ \| \nu \| \leq 1\} \backslash \{0\}} \nu(x)^2 H_1(x,d\nu),
  \end{align*}
  while the second part is estimated as follows:
  \begin{align*}
     \mathcal{R}_3^2(t \wedge \tau_m) 
     &\leq C_3(x) \left( \int_{ \{ \| \nu \| > 1\} } \nu(x) H_1(x,d\nu) \right) \int_0^{t \wedge \tau_m}|\eta_{s-}(x) - \xi_{s-}(x)| ds.
  \end{align*}
  For the fourth term, we obtain from \eqref{eq: estimate on phik}
  \begin{align*}
      \mathcal{R}_4(t\land\tau_m) &\leq \sum_{y\in V \setminus \{x\}}\int_0^{t\land\tau_m}\int_{\X \setminus \{0\}}\nu(x)|g(y,\eta_{s-}(y))-g(y,\xi_{s-}(y))|ds H_1(y,d\nu)
      \\ &\leq \sum_{y \in V \backslash \{x\}} \left( \int_{\X \backslash \{0\}} \nu(x) H_1(y,d\nu) \right) C_3(y) \int_0^{t \wedge \tau_m}|\eta_{s-}(y) - \xi_{s-}(y)| ds.
  \end{align*}
  Finally, we find that
  \begin{align*}
      \mathcal{R}_5(t \wedge \tau_m) &\leq \int_0^{t \wedge \tau_m} \int_{\X\backslash \{0\}} \int_{\R_+} |\Delta_1(x,s)| ds H_2(d\nu)du
      \\ &= \int_0^{t \wedge \tau_m} \int_{\X \backslash \{0\}} \nu(x) | \rho(x,\eta_{s-},\nu) - \rho(x,\xi_{s-}, \nu)| ds H_2(d\nu).
  \end{align*}
  This proves the desired inequalities for $\mathcal{R}_j(t\wedge \tau_m)$, $j = 1,\dots, 5$ and hence completes the proof of this statement.
\end{proof}

\section{Comparison principles}

In this section, we show that under conditions (A1) -- (A5), the process is monotone with respect to the initial condition. Moreover, we establish a monotonicity principle with respect to the drift parameters. Let us start with the following technical result. The desired comparison property is then proved afterwards.

\begin{Lemma}\label{lemma: monotonicity}
 Suppose that conditions (A1) and (A5) are satisfied.
 Then
 \[
  \sum_{x \in V}v(x)(B_0(x, \eta) - B_0(x,\xi))^+ \leq C_1(m)\sum_{x\in V} v(x)(\eta(x)-\xi(x))^+
 \]
 and 
 \begin{align*}
  &\ \int_{\X \backslash \{0\}} \sum_{x \in V} v(x)\nu(x)(\rho(x,\eta,\nu) - \rho(x,\xi,\nu))^+ H_2(d\nu) 
  \\ &\qquad \leq C_5(m)\sum_{x \in V}v(x)(\eta(x) - \xi(x))^+
 \end{align*}
 hold for all $\eta, \xi \in \X$ satisfying $\| \eta \|, \| \xi \| \leq m$ for some $m \in \N$ where $z^+ = \max\{z,0\}$.
\end{Lemma}
\begin{proof}
 For given $\eta, \xi \in \X$ we define 
 \[
    \min(\eta,\xi)(x) = \begin{cases}
                     \xi(x), &\text{if }\eta(x) \geq \xi(x)
                     \\
                     \eta(x), &\text{if }\eta(x) < \xi(x),
                   \end{cases} \qquad x \in V.
  \]
  Then $\min(\eta,\xi) \leq \xi$ and $\|\eta-\min(\eta,\xi)\| = \sum_{x\in V}v(x)(\eta(x)-\xi(x))^+$. Using (A1), we obtain 
   for  $\eta, \xi \in \X$ satisfying $\| \eta \|, \| \xi \| \leq m$
  \begin{align*}
    \sum_{x\in V}&v(x)(B_0(x,\eta)-B_0(x,\xi))^+
    \\
    &\leq \sum_{x\in V} v(x)(B_0(x,\eta)-B_0(x,\min(\eta,\xi)))^+ 
     +\sum_{x\in V}v(x)(B_0(x,\min(\eta,\xi))-B_0(x,\xi))^+    
    \\ &\leq \sum_{x\in V} v(x)|B_0(x,\eta)-B_0(x,\min(\eta,\xi))|
    \\ &\leq C_1(m)\|\eta-\min(\eta,\xi)\|
    \\ &= C_1(m)\sum_{x\in V} v(x)(\eta(x)-\xi(x))^+.
  \end{align*}
  Analogously, using (A5), we find that
  \begin{align*}
      &\ \int_{\X \backslash \{0\}} \sum_{x \in V} v(x)\nu(x)(\rho(x,\eta,\nu) - \rho(x,\xi,\nu))^+ H_2(d\nu) 
      \\ &\leq \int_{\X \backslash \{0\}} \sum_{x \in V} v(x)\nu(x)(\rho(x,\eta,\nu) - \rho(x,\min(\eta,\xi),\nu))^+ H_2(d\nu) 
      \\ &\ \ \ + \int_{\X \backslash \{0\}} \sum_{x \in V} v(x)\nu(x)(\rho(x,\min(\eta,\xi),\nu) - \rho(x,\xi,\nu))^+ H_2(d\nu) 
    \\ &\leq \int_{\X \backslash \{0\}} \sum_{x \in V} v(x)\nu(x)|\rho(x,\eta,\nu) - \rho(x,\min(\eta,\xi),\nu)| H_2(d\nu) 
    \\ &\leq C_5(m) \| \eta - \min(\eta,\xi)\|
    \\ &= C_5(m)\sum_{x \in E}V(x)(\eta(x) - \xi(x))^+.
  \end{align*}
  This proves the assertion.
\end{proof}
The following is the main result of this section.
\begin{Theorem}\label{comparison_ic}
 Suppose that conditions (A1) -- (A5) are satisfied. 
 Let $(\eta_t)_{t \geq 0}$ and $(\xi_t)_{t \geq 0}$ be two weak solutions to \eqref{SDE} defined on the same stochastic basis. Then for each $m \in \N$ and $t \geq 0$ it holds that
 \begin{align*}
  &\ \E\left[ \sum_{x \in V}v(x)(\eta_{t \wedge \tau_m}(x)-\xi_{t \wedge \tau_m}(x))^+ \right]
  \\ &\qquad \qquad \leq \E\left[ \sum_{x \in V}v(x)(\eta_0(x) - \xi_0(x))^+ \right]
  e^{(C_1(m) + 2C_4 + C_5(m))t},
 \end{align*}
 where $\tau_m := \tau_m(\eta) \wedge \tau_m(\xi)$
 is a sequence of stopping times with $\tau_m(\eta),\tau_m(\xi)$ defined as in Lemma \ref{lemma: localization}.
 In particular, if $\P[\eta_0 \leq \xi_0] = 1$,
 then $\P[ \eta_t \leq \xi_t, \ \ t \geq 0] = 1$.
\end{Theorem}
\begin{proof}
 Define $\zeta_t := \eta_t - \xi_t$ and fix $x \in V$. 
 Let $\phi_k: \R \longrightarrow \R_+$ be a sequence of twice continuously differentiable functions such that for each $k \geq 1$:
 \begin{enumerate} \label{paraphernalia}
     \item[(i)] $\phi_k(z) \nearrow z^+ := \max\{0,z\}$ as $k \to \infty$ for $z \geq 0$,
     \item[(ii)] $\phi_k(z) = \phi_k'(z) = \phi_k''(z) = 0$ for $z \leq 0$,
     \item[(iii)] $\phi_k'(z) \in [0,1]$ for $z \geq 0$,
     \item[(iv)] $\phi_k''(z)z \leq 2/k$ holds for all $z \geq 0$.
 \end{enumerate} 
 Note that the sequence in Thm. \ref{thm: uniqueness} approximates the absolute value function, while the function above approximates the rectified linear unit, had has been previously used in, e.g., \cite{FL10, Ma13, BLP15, FJR19b}. 
 
 To simplify the notation below, we set
 $D_h\phi_k(z) := \phi_k(z+h) - \phi_k(z)$ with $z,h \in \R$. Using the mean-value theorem one can check that \eqref{eq: estimate on phik} holds for all $z,h \in \R$.
 Applying the It\^{o} formula to $\zeta_t(x)$ gives
 \begin{align}\label{decomp_monotonicity}
 \phi_k(\zeta_t(x)) = \phi_k(\zeta_0(x)) 
     + \sum_{j=1}^5 \mathcal{R}_j(t) + \mathcal{M}(t),
 \end{align}
 where the processes $\mathcal{R}_1, \dots, \mathcal{R}_5, \mathcal{M}$ are given as in the proof of Theorem \ref{thm: uniqueness} and, in particular, also depend on the fixed value $x$.
 Let $\tau_m := \tau_m(\eta) \wedge \tau_m(\xi)$ with $\tau_m(\eta), \tau_m(\xi)$ defined in Lemma \ref{lemma: localization}. Then $\tau_m \longrightarrow \infty$ and \eqref{eq: boundedness by localization} holds. The same arguments as in the appendix prove that $(\mathcal{M}(t \wedge \tau_m))_{t \geq 0}$ is a martingale for each $k$ and each $m$. Below we will show that 
  \begin{align*}
   \mathcal{R}_1(t \wedge \tau_m) &\leq \int_0^{t \wedge \tau_m} (B_0(x,\eta_s)-B_0(x,\xi_s))^+ ds
   \\ \mathcal{R}_2(t \wedge \tau_m) &\leq C_2(x) \frac{2t}{k},
    \\ \mathcal{R}_3(t \wedge \tau_m) &\leq \frac{C_3(x)t}{k}\int_{ \{ \| \nu \| \leq 1\} \backslash \{0\}} \nu(x)^2 H_1(x,d\nu) 
   \\ &\qquad + C_3(x)\left( \int_{\{\| \nu \| > 1\}} \nu(x)H_1(x,d\nu) \right) \int_0^{t \wedge \tau_m}(\eta_{s-}(x) - \xi_{s-}(x))^+ ds,
   \\ \mathcal{R}_4(t \wedge \tau_m) &\leq \sum_{y \in V \backslash \{x\}} \left( \int_{\X \backslash \{0\}} \nu(x) H_1(y,d\nu) \right) C_3(y) \int_0^{t \wedge \tau_m}(\eta_{s-}(y) - \xi_{s-}(y))^+ ds,
   \\ \mathcal{R}_5(t \wedge \tau_m) &\leq \int_0^{t \wedge \tau_m} \int_{\X \backslash \{0\}} \nu(x)  (\rho(x,\eta_{s-},\nu) - \rho(x,\xi_{s-}, \nu))^+ ds H_2(d\nu)
  \end{align*}
  Taking then expectations in \eqref{decomp_monotonicity} and using the above estimates gives
  \begin{align*}
      \E[ \phi_k(\zeta_{t \wedge \tau_m}(x))]
      &= \E[ \phi_k(\zeta_0(x)) ]
      + \E \left[ \sum_{j=1}^{5}\mathcal{R}_j(t \wedge \tau_m) \right]
      \\
      &\leq \E[ \phi_k(\zeta_0(x)) ]
      + \E\left[ \int_0^{t \wedge \tau_m} (B_0(x,\eta_s) - B_0(x,\xi_s))^+ ds \right]
      \\
      &\ \ \ + C_2(x) \frac{2t}{k}
      + \frac{C_3(x)t}{k}\int_{ \{ \| \nu \| \leq 1\} \backslash \{0\}} \nu(x)^2 H_1(x,d\nu) 
      \\ &\ \ \ + C_3(x)\left( \int_{\{\| \nu \| > 1\}} \nu(x)H_1(x,d\nu) \right) \E\left[ \int_0^{t \wedge \tau_m}(\eta_{s-}(x) - \xi_{s-}(x))^+ ds \right]
      \\ &\ \ \ + \sum_{y \in V \backslash \{x\}} \left( \int_{\X \backslash \{0\}} \nu(x) H_1(y,d\nu) \right) C_3(y) \E\left[\int_0^{t \wedge \tau_m}(\eta_{s-}(y) - \xi_{s-}(y))^+ ds \right]
      \\ &\ \ \ + \E\left[\int_0^{t \wedge \tau_m} \int_{\X \backslash \{0\}} \nu(x) (\rho(x,\eta_{s-}, \nu) - \rho(x,\xi_{s-},\nu))^+ H_2(d\nu)ds \right].
  \end{align*}
  Letting first $k \to \infty$ and then taking the $v$-weighted sum over  $x\in V$ we get by (i) 
  \begin{align*}
      \E&\left[\sum_{x \in V}v(x) (\eta_{t \wedge \tau_m}(x) - \xi_{t \wedge \tau_m}(x))^+\right]
      \\ &\leq \sum_{x\in V} v(x) \E\left[(\eta_0(x)-\xi_0(x))^+\right] 
      + \E\left[\int_0^{t \wedge \tau_m}\sum_{x\in V} v(x)(B_0(x,\eta_s)-B_0(x,\xi_s))^+ ds\right]
      \\ &\ \ \ + 2C_4 \E\left[\int_0^{t \wedge \tau_m} \sum_{x\in V} v(x)(\eta_{s-}(x)-\xi_{s-}(x))^+ ds\right]
      \\ &\ \ \ +\E\left[\int_0^{t \wedge \tau_m}\int_{\X \setminus\{0\}}\sum_{x\in V}v(x)\nu(x)(\rho(x,\eta_{s-},\nu) - \rho(x,\xi_{s-},\nu))^+ H_2(d\nu) ds \right]
      \\ &\leq \E\left[\sum_{x\in V} v(x) (\eta_0(x)-\xi_0(x))^+\right]
      \\ &\ \ \ + \left(C_1(m)+2C_4+C_5(m)\right)\E\left[\int_0^{t \wedge \tau_m}\sum_{x\in V} v(x)(\eta_{s-}(x)-\xi_s(x))^+ ds\right]
  \end{align*}
  where we have used Lemma \ref{lemma: monotonicity}. The assertion follows from the Gronwall lemma. Hence it remains to prove the estimates for $\mathcal{R}_j(t \wedge \tau_m)$, $j = 1,\dots, 5$.
  
  The first estimate above follows directly by the properties of $\phi^\prime$. Indeed, it follows from (A1) and (ii) that $\phi_k'(\zeta_s(x))(B_1(x,\eta_s(x)) - B_1(x,\xi_s(x))) \geq 0$ holds a.s. for $s \in [0,t \wedge \tau_m]$. Thus we obtain
  \begin{align*}
  \mathcal{R}_1(t \wedge \tau_m) & =   \int_0^{t \wedge \tau_m} \phi_k'(\zeta_{s-}(x))(B(x,\eta_s) - B(x,\xi_s))ds
   \\ &= \int_0^{t \wedge \tau_m} \phi_k'(\zeta_{s-}(x))(B_0(x,\eta_s) - B_0(x,\xi_s))ds
   \\ &\ \ \ - \int_0^{t \wedge \tau_m} \phi_k'(\zeta_{s-}(x))(B_1(x,\eta_s(x)) - B_1(x,\xi_s(x)))ds
   \\ &\leq \int_0^{t \wedge \tau_m} \phi_k'(\zeta_{s-}(x))(B_0(x,\eta_s) - B_0(x,\xi_s))ds
   \\ &\leq \int_0^{t \wedge \tau_m}(B_0(x,\eta_s) - B_0(x,\xi_s))^+ds.
  \end{align*}
  The desired estimates for $\mathcal{R}_2, \mathcal{R}_3$, and $\mathcal{R}_4$ can be shown in exactly the same way as in the proof of Theorem \ref{thm: uniqueness}. Let us now consider the term $\mathcal{R}_5$. By (ii) we have 
  \begin{align*}
       &\ \int_0^{\infty} D_{\Delta_1(x,s)}\phi_k(\zeta_{s-}(x)) du
       \\ &\leq \int_{0}^{\infty} \1_{ \{ \rho(x,\eta_{s-},\nu) \geq \rho(x,\xi_{s-}, \nu) \} } D_{\Delta_1(x,s)}\phi_k(\zeta_{s-}(x)) du 
       \\ &\leq \int_{\rho(x,\xi_{s-},\nu)}^{\rho(x,\eta_{s-},\nu)}\1_{ \{ \rho(x,\eta_{s-},\nu) \geq \rho(x,\xi_{s-}, \nu) \} } \nu(x)du
      \\ &= (\rho(x, \eta_{s-},\nu) - \rho(x,\xi_{s-},\nu))^+ \nu(x).
  \end{align*}
  This implies the desired estimate for $\mathcal{R}_5(t\wedge \tau_m)$.
  Hence we have shown all the desired inequalities for $\mathcal{R}_j(t\wedge \tau_m)$, $j = 1,\dotsc, 5$ and the proof of the theorem is complete.
\end{proof}

Theorem \ref{comparison_ic} can be generalized to the case when $(\eta_t)_{t \geq 0}$ and $(\xi_t)_{t \geq 0}$ are solutions to equations with different 
functions $ B,  B_0,  B_1,  c,  g, \rho$. Here we only give a simple version with different drifts. 
\begin{Corollary}\label{thm comparison different drift}
Let $(B, B_0, B_1, c, g, \rho)$ and $(\widetilde B,\widetilde  B_0, \widetilde B_1, c,  g,  \rho)$ be $C_{\overline{1,6}}$-admissible tuples. On the same stochastic basis let $(\eta_t)_{t \geq 0}$ satisfy \eqref{SDE} and let $(\xi_t)_{t \geq 0}$ satisfy 
\begin{align}\label{SDE different drift}
    \xi_t(x) &= \xi_0(x) + \int_0^t \widetilde B(x,\xi_s)ds 
    + \int_0^t \sqrt{2c(x,\xi_s(x))}dW_s(x)
    \\ \notag &\ \ \ + \int_0^t \int_{\X \backslash \{0\}}\int_{\R_+} \nu(x) \mathbbm{1}_{ \{ u \leq g(x,\xi_{s-}(x))\} } \widetilde{N}_x(ds,d\nu,du)
    \\ \notag &\ \ \ + \sum_{y \in V \backslash \{x\}}\int_0^t \int_{\X \backslash \{0\}}\int_{\R_+} \nu(x) \mathbbm{1}_{ \{ u \leq g(y,\xi_{s-}(y))\} } N_y(ds,d\nu,du)
    \\ \notag &\ \ \ + \int_{0}^t \int_{\X \backslash \{0\}} \int_{\R_+}\nu(x) \mathbbm{1}_{ \{ u \leq \rho(x,\xi_{s-},\nu)\} } M(ds,d\nu,du).
\end{align}
 Assume that for all $\alpha, \beta \in \X$, $\alpha \leq \beta$, and $x \in V$ one has $B(x, \alpha) \leq \widetilde B (x, \beta)$. Then for each $m \in \N$ and $t \geq 0$ it holds that
 \begin{align*}
  &\ \E\left[ \sum_{x \in V}v(x)(\eta_{t \wedge \tau_m}(x)-\xi_{t \wedge \tau_m}(x))^+ \right]
  \\ &\qquad \qquad \leq \E\left[ \sum_{x \in V}v(x)(\eta_0(x) - \xi_0(x))^+ \right]
  e^{(C_1(m) + 2C_4 + C_5(m))t},
 \end{align*}
 where $\tau_m := \tau_m(\eta) \wedge \tau_m(\xi)$
 is a sequence of stopping times with $\tau_m(\eta),\tau_m(\xi)$ defined as in Lemma \ref{lemma: localization}. In particular, if $\P[\eta_0 \leq \xi_0] = 1$, then $\P[ \eta_t \leq \xi_t, \  t \geq 0] = 1$.
\end{Corollary}
\begin{proof}
    The proof follows the same steps as the proof of Theorem \ref{comparison_ic}, the only difference is that 
    with these settings 
    \[
     \mathcal{R}_1(t) = \int_0^t \phi_k'(\zeta_s(x))\left( B(x,\eta_s) - \widetilde B(x,\xi_s) \right)ds,
    \]
    so we only need to note that 
    \[
    \int_0^t \phi_k'(\zeta_s(x))\left( B(x,\eta_s) - \widetilde B(x,\xi_s) \right)ds
    \leq 
    \int_0^t \phi_k'(\zeta_s(x))\left( B(x,\eta_s) -   B(x,\xi_s) \right)ds.
    \]
\end{proof}

Finally, we formulate an auxiliary comparison principle used for the construction of solutions of \eqref{SDE}. 
\begin{Theorem}\label{thm special comparison}
Let $|V| < \infty$ and let $(\widetilde B, \widetilde B_0, \widetilde B_1, c, \widetilde g, \widetilde \rho)$  
 be a $C_{\overline{1,6}}$-admissible tuple
and  for $ V' \subset V$ let 
 \begin{align}
     B(x, \alpha)   & = \widetilde   B (x, \alpha) \1\{x \in V' \}
    \label{Wagnis = venture1}
     \\
       g(x, \alpha (x))& =  \widetilde g(x, \alpha (x))\1\{x \in V' \},
           \\
        \rho (x, \alpha, \nu) & = \widetilde \rho (x, \alpha, \nu)\1\{x \in V' \}.
      \label{Wagnis = venture3}
 \end{align}
 On the same stochastic basis let $(\eta_t)_{t \geq 0}$ satisfy \eqref{SDE}
 and let $(\xi_t)_{t \geq 0}$ satisfy 
\begin{align}\label{SDE special comparison}
    \xi_t(x) &= \xi_0(x) + \int_0^t \widetilde B(x,\xi_s)ds 
    + \int_0^t \sqrt{2c(x,\xi_s(x))}dW_s(x)
    \\ \notag &\ \ \ + \int_0^t \int_{\X \backslash \{0\}}\int_{\R_+} \nu(x) \mathbbm{1}_{ \{ u \leq \widetilde g(x,\xi_{s-}(x))\} } \widetilde{N}_x(ds,d\nu,du)
    \\ \notag &\ \ \ + \sum_{y \in V \backslash \{x\}}\int_0^t \int_{\X \backslash \{0\}}\int_{\R_+} \nu(x) \mathbbm{1}_{ \{ u \leq \widetilde g(y,\xi_{s-}(y))\} } N_y(ds,d\nu,du)
    \\ \notag &\ \ \ + \int_{0}^t \int_{\X \backslash \{0\}} \int_{\R_+}\nu(x) \mathbbm{1}_{ \{ u \leq \widetilde \rho(x,\xi_{s-},\nu)\} } M(ds,d\nu,du),
\end{align}
 Assume further that $\P[\eta_0 \leq \xi_0] = 1$ and for $x \in V \setminus V' $, $\P[ \eta_0(x) = 0] = 1$. Then $\P[ \eta_t \leq \xi_t, \  t \geq 0] = 1$.
\end{Theorem}
\begin{proof}
Since $(\widetilde B, \widetilde B_0, \widetilde B_1, c, \widetilde g, \widetilde \rho)$ is $C_{\overline{1,6}}$-admissible, the tuple
$( B,  B_0,   B_1, c,  g,  \rho)$ is a $C_{\overline{1,6}}$-admissible as well by \eqref{Wagnis = venture1}-\eqref{Wagnis = venture3}. For $x  \in V \setminus V'$ we have a.s. $\eta _t (x) = 0$ and hence $\P[ \eta_t(x) \leq \xi_t(x), \  t \geq 0] = \P[ 0 \leq \xi_t(x), \  t \geq 0]= 1$. Let $\{ \phi_k\} _{k \in \N}$ be the sequence of functions introduced in the proof of Theorem \ref{comparison_ic}. Set $\zeta_t := \eta_t - \xi_t$. Recall the notation $D_h\phi_k(z) := \phi_k(z+h) - \phi_k(z)$ for $z,h \in \R$.

 For $x \in V'$ by the the It\^{o} formula
 \begin{align}\label{Ito to zeta}
 \phi_k(\zeta_t(x)) = \phi_k(\zeta_0(x)) 
     + \sum_{j=1}^5 \mathcal{D}_j(t) + \mathcal{M}(t),
 \end{align}
 where 
 \begin{align*}
 \mathcal{D}_1(t) &= \int_0^t \phi_k'(\zeta_s(x))\left( B(x,\eta_s) - \widetilde B(x,\xi_s) \right)ds
 \\ \mathcal{D}_2(t) &= \frac{1}{2}\int_0^t \phi_k''(\zeta_s(x))\left( \sqrt{2c(x,\eta_s(x))} - \sqrt{2 c(x,\xi_s(x))}\right)^2 ds
 \\ \mathcal{D}_3(t) &= \int_0^t \int_{\X \backslash \{0\}} \int_{\R_+} \left( D_{\widetilde \Delta_0(x,s)}\phi_k(\zeta_{s-}(x)) - \phi_k'(\zeta_{s-}(x))\widetilde \Delta_0(x,s) \right) ds H_1(x,d\nu)du
  \\ \mathcal{D}_4(t) &= \sum_{y \in V \backslash \{x\}}\int_0^t \int_{\X \backslash \{0\}} \int_{\R_+} D_{\widetilde \Delta_0(y,s)}\phi_k(\zeta_{s-}(x)) ds H_1(y,d\nu)du
 \\ \mathcal{D}_5(t) &= \int_0^t \int_{\X \backslash \{0\}}\int_{\R_+} D_{\widetilde \Delta_1(x,s)}\phi_k(\zeta_{s-}(x)) ds H_2(d\nu)du
 \end{align*}
  with increments given by 
 \begin{align*}
     \widetilde \Delta_0(z,s) &= \nu(x) \left(\mathbbm{1}_{ \{ u \leq g(z,\eta_{s-}(z))\} } - \mathbbm{1}_{ \{ u \leq \widetilde  g(z,\xi_{s-}(z))\} }\right),
     \\ \widetilde \Delta_1(z,s) &= \nu(z) \left(\mathbbm{1}_{ \{ u \leq \rho(z,\eta_{s-},\nu)\} } - \mathbbm{1}_{ \{ u \leq \widetilde  \rho(z,\xi_{s-},\nu)\} } \right)
 \end{align*}
 and  
 \begin{align*}
     \mathcal{M}(t) &= \int_0^t \phi_k'(\zeta_s(x)) \left( \sqrt{2c(x,\eta_s(x))} - \sqrt{2 c(x,\xi_s(x))}\right)dW_s(x)
     \\ \notag &\ \ \ + \sum_{y \in V}\int_0^t \int_{\X \backslash \{0\}} \int_{\R_+}  D_{\widetilde \Delta_0(y,s)}\phi_k(\zeta_{s-}(x)) \widetilde{N}_y(ds,d\nu,du)
     \\ \notag &\ \ \ + \int_0^t \int_{\X \backslash \{0\}}\int_{\R_+} D_{\widetilde \Delta_1(x,s)}\phi_k(\zeta_{s-}(x)) \widetilde{M}(ds,d\nu,du).
 \end{align*}
 The process $(\mathcal{M}(t), t\geq 0)$ is a local martingale. 
For $x \in V'$, we use \eqref{Wagnis = venture1}-\eqref{Wagnis = venture3} to find that
\begin{equation}\label{die Narbe = scar}
     \mathcal{D}_j(t) =  \mathcal{R}_j(t),
     \ \ \ t \geq 0, j = 1,2,3,5.
\end{equation}
 where $\mathcal{R}_j(t)$ are given as in the proof of Theorem \ref{thm: uniqueness}. For $\mathcal{D}_4(t)$ we write 
    \begin{align*}
         \mathcal{D}_4(t) &= \sum_{y \in V  \setminus V'}\int_0^t \int_{\X \backslash \{0\}} \int_{\R_+} \big[ \phi_k(\zeta_{s-}(x) +  \widetilde \Delta_0(y,s))
          - \phi_k(\zeta_{s-}(x) ) \big]ds H_1(y,d\nu)du
          \\ &\qquad + \sum_{y \in V ' \backslash \{x\} }\int_0^t \int_{\X \backslash \{0\}} \int_{\R_+} \big[ \phi_k(\zeta_{s-}(x) +  \widetilde \Delta_0(y,s))
          - \phi_k(\zeta_{s-}(x) ) \big]ds H_1(y,d\nu)du
    \end{align*}
    Recall that $\Delta _0, \Delta _1$ were introduced in the proof of Theorem \ref{thm: uniqueness}. For $y \in V \setminus V'$ for $u > 0$
\begin{align*}
   \widetilde \Delta_0(y,s) = -\nu(x)  \mathbbm{1}_{ \{ u \leq \widetilde  g(y,\xi_{s-}(y))\} } \leq 0= 
\Delta_0(y,s) 
\end{align*}
whereas for $y \in V ' \backslash \{x\}$ we have $\widetilde \Delta_0(y,s) = \Delta_0(y,s)$. Since $\phi_k$ is non-decreasing we arrive at 
    \begin{align*}
         \mathcal{D}_4(t) &  \leq  \sum_{y \in V  \setminus V'}\int_0^t \int_{\X \backslash \{0\}} \int_{\R_+} \big[ \phi_k(\zeta_{s-}(x) +    \Delta_0(y,s))
          - \phi_k(\zeta_{s-}(x) ) \big]ds H_1(y,d\nu)du
          \\
          & + \sum_{y \in V ' \backslash \{x\} }\int_0^t \int_{\X \backslash \{0\}} \int_{\R_+} \big[ \phi_k(\zeta_{s-}(x) +  \Delta_0(y,s))
          - \phi_k(\zeta_{s-}(x) ) \big]ds H_1(y,d\nu)du
          \\ &= \mathcal{R} _4(t).
    \end{align*}
    Combining this with \eqref{Ito to zeta} and \eqref{die Narbe = scar}
    we get 
     \begin{align*} 
 \phi_k(\zeta_t(x)) \leq \phi_k(\zeta_0(x)) 
     + \sum_{j=1}^5 \mathcal{R}_j(t) + \mathcal{M}(t).
 \end{align*}
  
From here the proof goes in exactly the same path
as the proof of Theorem \ref{comparison_ic}
follows from \eqref{decomp_monotonicity}. The fact that here we have an inequality instead of an equality in \eqref{decomp_monotonicity} does not make a difference.
 \end{proof}

\section{Construction of a weak solution}

Firstly we study the case where $V$ is finite. In such a case $\X = \R_+^{|V|}$ and we take $v(x) = 1$ so that $\|\eta\| = \sum_{k = 1}^{|V|}|\eta_k|$ corresponds to the $1$-norm on $\R^{|V|}$. In this case, \eqref{SDE} becomes a classical SDE for which we may use existing results on the existence of weak solutions. The precise statement is summarized in the next lemma. 
\begin{Lemma}\label{lemma: finite E}
 Suppose that $V$ is a finite set and that conditions (A1) -- (A6) are satisfied for $v(x) = 1$. Then for each $\eta_0$ being $\F_0$-measurable with $\E[\|\eta_0\|]<\infty$, \eqref{SDE} has a unique strong solution in $\X = \R_+^{|V|}$.
\end{Lemma}
\begin{proof}
 It follows from \cite{MR4350978} that for each $n \geq 1$ the equation
 \begin{align}\label{SDE finite E}
    \eta_t(x) &= \eta_0(x) + \int_0^t B(x,\eta^+_s)ds 
    + \int_0^t \sqrt{2c(x,\eta_s^+(x))}dW_s(x)
    \\ \notag &\ \ \ + \int_0^t \int_{\X \backslash \{0\}}\int_{\R_+} \1_{\{\|\nu\| \leq n\}}\nu(x) \mathbbm{1}_{ \{ u \leq g(x,\eta_{s-}^+(x))\} } \widetilde{N}_x(ds,d\nu,du)
    \\ \notag &\ \ \ + \sum_{y \in V \backslash \{x\}}\int_0^t \int_{\X \backslash \{0\}}\int_{\R_+} \1_{\{\|\nu\| \leq n\}}\nu(x) \mathbbm{1}_{ \{ u \leq g(y,\eta_{s-}^+(y))\} } N_y(ds,d\nu,du)
    \\ \notag &\ \ \ + \int_{0}^t \int_{\X \backslash \{0\}} \int_{\R_+}\1_{\{\|\nu\| \leq n\}}\nu(x) \mathbbm{1}_{ \{ u \leq \rho(x,\eta_{s-}^+,\nu)\} } M(ds,d\nu,du)
 \end{align}
 with $\eta^+(y) = \max\{0,\eta(y)\}$ for $y \in V$ has a weak solution on $\X = \R^{|V|}$. Since the coefficients $c^+(x,t) = c(x,t^+)$, $g^+(x,t) = g(x,t^+)$, and $\rho^+(x,t,\nu) = \rho(x,t^+,\nu)$ still satisfy the conditions (A1) -- (A6), in view of Theorem \ref{thm: uniqueness}, also pathwise uniqueness holds and hence this solution is strong. By following the argument given in \cite[Section 2]{FL10}, we prove that this solution is nonnegative.  
 Suppose that there exists $\e > 0$ and $x \in V$ such that  
 $\tau = \inf\{ t \geq 0 \ : \ \eta_t(x) \leq -\e \}$ satisfies $\P[\tau < \infty] > 0$. Then $\eta_{\tau}(x) = \eta_{\tau-}(x) \leq -\e$ holds on $\{\tau < \infty\}$.
 Let $\sigma = \inf\{ s \in (0,\tau) \ : \ \eta_t(x) \leq 0, \ \ \forall t \in [s,\tau]\}$. Then $\sigma < \tau$ a.s., and hence we can find a deterministic time $r \geq 0$ such that $\{\sigma \leq r < \tau \}$ has positive probability. A.s. on this event, we find for $t \geq r$
 \begin{align*}
    \eta_{t \wedge \tau}(x) &= \eta_{r \wedge \tau}(x) + \int_{r \wedge \tau}^{t \wedge \tau} B(x,\eta_s^+)ds 
    \\ \notag &\ \ \ + \sum_{y \in E \backslash \{x\}}\int_{r \wedge \tau}^{t \wedge \tau} \int_{\X \backslash \{0\}}\int_{\R_+} \nu(x) \mathbbm{1}_{ \{ u \leq g(y,\eta_{s-}^+(y))\} } N_y(ds,d\nu,du)
    \\ \notag &\ \ \ + \int_{r \wedge \tau}^{t \wedge \tau} \int_{\X \backslash \{0\}} \int_{\R_+}\nu(x) \mathbbm{1}_{ \{ u \leq \rho(x,\eta_{s-}^+,\nu)\} } M(ds,d\nu,du).
 \end{align*}
 In view of condition (A1) we have $B(x,\eta) \geq 0$ whenever $\eta \in \X$ is such that $\eta(x) = 0$. Thus $t \longmapsto \eta_{t \wedge \tau}(x)$ is non-decreasing. Since $\eta_{r}(x) > -\e$ on $\{r < \tau\}$, we get a contradiction to $\eta_{\tau}(x) = \eta_{\tau-}(x) \leq -\e$. Hence the solution is nonnegative.
 
 It remains to show that we can pass to the limit $n \to \infty$. This procedure is rather standard, so we only provide a sketch of the proof. Let $(\eta_t^n)_{n \geq 1}$ be the unique strong solution of \eqref{SDE finite E}. It is not difficult to show that
 \[
  \sup_{n \geq 1}\E\left[ \sup_{t \in [0,T]} |\eta_t^{(n)}(x)| \right] < \infty, \qquad \forall x \in V.
 \]
 Hence using the Aldous criterion, we find that the sequence of processes $(\eta_t^n)_{n \geq 1}$ is tight on the Skorohod space. Using convergence of the martingale problems, we may show that any of its limits is a weak solution of \eqref{SDE} (with $|V| < \infty$). This completes the proof.
\end{proof}

In the second step,  we use Lemma \ref{lemma: finite E} to approximate a weak solution via $V_N \nearrow V$ where $V_N$ is an increasing sequence of finite sets.

\begin{Theorem}\label{thm weak existence}
 Suppose that conditions (A1) -- (A6) are satisfied. Then weak existence holds for \eqref{SDE} and any $\F_0$-measurable initial condition $\eta_0 \in \X$ satisfying $\E[\|\eta_0\|] < \infty$.
\end{Theorem}
\begin{proof}
 \textit{Step 1.} Fix any stochastic basis $(\Omega, \F, (\F_t)_{t \geq 0}, \P)$, let noise terms given as in (N1) -- (N4), and let $\eta_0$ be an $\F_0$-measurable random variable with $\E[\|\eta_0\|] < \infty$. Let $(V_N)_{N\in \N}$ be a sequence of finite sets in $V$ such that $V_N \nearrow V$. Define $B^N(x,\eta) = \1_{V_N}(x)B_0(x,\eta) - B_1(x,\eta(x))$, $g^N(x,t) = \1_{V_N}(x)g(x,t)$, and $\rho^N(x,\eta,\nu) = \1_{V_N}(x)\rho(x,\eta,\nu)$. Then conditions (A1) -- (A6) are still satisfied with $V$ replaced by $V_N$, and \eqref{SDE} takes for these restricted coefficients the form
  \begin{align*}
    \eta_t^N(x) &= \eta^N_0(x) + \int_0^t B^N(x,\eta^N_s)ds 
    + \int_0^t \sqrt{2c(x,\eta^N_s(x))}dW_s(x)
    \\ \notag &\ \ \ + \int_0^t \int_{\X \backslash \{0\}}\int_{\R_+} \nu(x) \mathbbm{1}_{ \{ u \leq g^N(x,\eta^N_{s-}(x))\} } \widetilde{N}_x(ds,d\nu,du)
    \\ \notag &\ \ \ + \sum_{y \in V \backslash \{x\}}\int_0^t \int_{\X \backslash \{0\}}\int_{\R_+} \nu(x) \mathbbm{1}_{ \{ u \leq g^N(y,\eta^N_{s-}(y))\} } N_y(ds,d\nu,du)
    \\ \notag &\ \ \ + \int_{0}^t \int_{\X \backslash \{0\}} \int_{\R_+} \nu(x) \mathbbm{1}_{ \{ u \leq \rho^N(x,\eta^N_{s-},\nu)\} } M(ds,d\nu,du)
\end{align*}
where $\eta_0^N$ is defined by $\eta_0^N(x) = \1_{V_N}(x)\eta_0(x)$. Thus the equation is effectively an equation for $\eta_t^N(x)$ with $x \in V_N$ which has a unique strong solution due to Lemma \ref{lemma: finite E}.
 
 \textit{Step 2.} Using $B^N \leq B^{N+1}$ and $\rho^N \leq \rho^{N+1}$ and the comparison principle in Theorem \ref{thm special comparison}, we find that $\P[ \eta_t^N \leq \eta_t^{N+1}, \ \ t \geq 0] = 1$ for $N \geq 1$. Since $B^N\leq B$ and $\rho^N\leq\rho$ for all $N\in\N$, we may apply Theorem \ref{first_moment_estimate} with $C$ independent of $N$ to show that 
 \[
  \sup_{t \in [0,T]}\sup_{N \geq 1}\E[ \| \eta_t^N\| ] < \infty, \qquad T > 0.
 \]
 Define a new process $(\eta_t)_{t \geq 0}$ by $\eta_t(x) = \sup_{N \geq 1} \eta_t^N(x)$ for $x \in V$. Then $(\eta_t)_{t \geq 0}$ is $(\F_t)_{t \geq 0}$-adapted and by monotone convergence, we see that
 \begin{equation*}
  \sup_{t \in [0,T]}\E[ \|\eta_t\| ] \leq \sup_{t \in [0,T]}\sup_{N \geq 1}\E[ \| \eta_t^N\| ] < \infty,
 \end{equation*}
 i.e., $(\eta_t)_{t \geq 0}$ takes values in $\X$. Note that by monotone convergence, we also have
 \begin{equation}\label{dominated_convergence}
  \lim_{N \to \infty}\int_0^T \E\left[\|\eta_t-\eta_t^N\| \right] dt = 0.
 \end{equation}
 
 \textit{Step 3.} The arguments in Step 2 already infer that the process $(\eta_t)_{t\geq 0}$ is $\X$-valued. Therefore, it remains to show that $(\eta_t)_{t \geq 0}$ is a solution to \eqref{SDE}. We consider all terms of \eqref{SDE} separately.  Convergence of the initial conditions, i.e. 
 $\lim_{N\to\infty}\eta_0^N(x)=\eta_0(x)$ is clear. For the drift we obtain
 \begin{align*}
     &\ \E\left[\left| \int_0^t\mathbbm{1}_{V_N}(x)B(x,\eta_s^N)ds - \int_0^t B(x,\eta_s)ds \right| \right]
     \\ &\leq \int_0^t \E\left[ |B_0(x,\eta_s^N) - B_0(x,\eta_s)|\right] ds 
     + \int_0^t \E \left[ |B_1(x,\eta_s^N(x)) - B_1(x,\eta_s(x))| \right] ds
     \\ &\ \ \ + \1_{V_N^c}(x) \int_0^t \E\left[ |B(x,\eta_s)|\right] ds.
 \end{align*}
 The first two terms converge to zero as $N \to \infty$ by monotone convergence being applicable due to condition (A1). The last term is finite due to the boundedness of the first moment and hence converges to zero for fixed $x \in V$. For the continuous noise part, we obtain
 \begin{align*}
  &\E \left[\left| \int_0^t \sqrt{2 c(x,\eta_s(x))} dW_s(x) - \int_0^t \sqrt{2 c(x,\eta_s^N(x))} dW_s(x)\right|^2 \right]
  \\ &= \int_0^t \E\left[ \left(\sqrt{2c(x,\eta_s(x))} - \sqrt{2c(x,\eta_s^N(x))} \right)^2 \right]ds
  \\ &\leq 2 \int_0^t \E\left[| c(x,\eta_s(x)) - \1_{V_N}(x)c(x,\eta_s^N(x))|\right] ds
   \\ &\leq 2 \int_0^t \E\left[ 1_{\{\eta_s(x) \leq R\}} |c(x,\eta_s(x)) - c(x,\eta_s^N(x))| \right] ds
     \\ &\ \ \ + 2 \int_0^t \E\left[ \1_{\{\eta_s(x) > R\}}|c(x,\eta_s(x)) - c(x,\eta_s^N(x))| \right] ds
     \\ &\leq 2 C_2(x)  \E\left[ |\eta_s(x) - \eta_s^N(x)| \right] ds
     + 4C_2(x)  \int_0^t\E\left[ 1_{\{\eta_s(x) > R\}}\eta_s(x)\right]ds
 \end{align*}
 where we have used $\eta_s^N(x) \leq \eta_s(x)$.
 Thus, taking first for fixed $R$ the limit $N \to \infty$, and then letting $R \to \infty$, proves the convergence to zero. 

 For the stochastic integrals against $\widetilde{N}_x$, we split the integrals into $\{\|\nu\|\leq 1\}\setminus\{0\}$ and $\{\|\nu\| > 1\}$ and study them separately. First, note that using It\^o's isometry (e.g. \cite[p. 63]{ikeda_watanabe}),
 \begin{align*}
     &\E\Big[\Big(\int_0^t\int_{\{\|\nu\|\leq 1\}\setminus\{0\}}\int_{\R_+}\nu(x)\mathbbm{1}_{\{u\leq g(x,\eta_{s-}(x))\}}\tilde{N}_x(ds,d\nu,du)
     \\
     &\hspace{50pt}-\int_0^t\int_{\{\|\nu\|\leq 1\}\setminus\{0\}}\int_{\R_+}\nu(x)\mathbbm{1}_{\{u\leq g(x,\eta^N_{s-}(x))\mathbbm{1}_{V_N}(x)\}}\tilde{N}_x(ds,d\nu,du)\Big)^2\Big]
     \\
     &=\int_0^t\int_{\{\|\nu\|\leq 1\}\setminus\{0\}}\nu(x)^2\E\Big[\int_{\R_+}|\mathbbm{1}_{\{u\leq g(x,\eta_{s-}(x))\}}-\mathbbm{1}_{\{u\leq g(x,\eta^N_{s-}(x))\}}\mathbbm{1}_{V_N}(x)|du\Big]H_1(x,d\nu)ds
     \\
     &\leq\mathbbm{1}_{V_N}(x)\int_0^t\int_{\{\|\nu\|_V\leq 1\}\setminus\{0\}}\nu(x)^2\E|g(x,\eta_{s-}(x))-g(x,\eta^N_{s-}(x))|H_1(x,d\nu)ds
     \\
     &\hspace{50pt}+\mathbbm{1}_{E_N^c}(x)\int_0^t\int_{\{\|\nu\|\leq 1\}\setminus\{0\}}\nu(x)^2\E[g(x,\eta_{s-}(x)]H_1(x,d\nu)ds
     \\
     &\leq C_3(x)\int_0^t\int_{\{\|\nu\|\leq 1\}\setminus\{0\}}\nu(x)^2\E|\eta_{s-}(x)-\eta^N_{s-}(x)|H_1(x,d\nu)
     \\
     &\hspace{50pt}+\mathbbm{1}_{V_N^c}(x)C_3(x)\int_0^t\int_{\{\|\nu\|\leq 1\}\setminus\{0\}}\nu(x)^2\E(\eta_{s-}(x))H_1(x,d\nu)ds
 \end{align*}
 where we used (A3) in the end. The first term tends to zero as $N\to \infty$ due to \eqref{dominated_convergence}, while the second one due to the indicator function.
 For the integrals against $\{\|\nu\| > 1\}$ we find that 
 \begin{align*}
     &\frac{1}{2}\E\Big[\Big|\int_0^t\int_{\{\|\nu\|>1\}}\int_{\R_+}\nu(x)\mathbbm{1}_{\{u\leq g(x,\eta_s(x))\}}\tilde{N}_x(ds,d\nu,du)
     \\ &\hspace{50pt}-\int_0^t\int_{\{\|\nu\|>1\}}\int_{\R_+}\nu(x)\mathbbm{1}_{\{u\leq g(x,\eta^N_s(x))\}}\mathbbm{1}_{V_N}(x)\tilde{N}_x(ds,d\nu,du)\Big|\Big]
     \\ &\leq \int_0^t\int_{\{\|\nu\|>1\}}\nu(x)\E\Big[\int_{\R_+}\big|\mathbbm{1}_{\{u\leq g(x,\eta_s(x))\}}-\mathbbm{1}_{\{u\leq g(x,\eta^N_s(x))\}}\mathbbm{1}_{V_N}(x)\big|du\Big]H_1(x,d\nu)ds
    \\ &\leq \mathbbm{1}_{V_N^c}(x)\int_0^t\int_{\{\|\nu\|>1\}}\nu(x)\E\big[g(x,\eta_s(x))\big]H_1(x,d\nu)ds
  \\ &\hspace{10pt}+\mathbbm{1}_{V_N}(x)\int_0^t\int_{\{\|\nu\|>1\}}\nu(x)\E\big[\big|g(x,\eta_s(x))-g(x,\eta_s^N(x))\big|\big]H_1(x,d\nu)ds
  \\ &=\mathbbm{1}_{V_N^c}(x)C_3(x)\int_{\{\|\nu\|>1\}}\nu(x)H_1(x,d\nu)\int_0^t\E[\eta_s(x)]ds
  \\
  &\hspace{10pt}+ C_3(x)\int_{\{\|\nu\|>1\}}\nu(x)H_1(x,d\nu)\int_0^t\E\big|\eta_s(x)-\eta_s^N(x)\big|ds
 \end{align*}
 Also here the right-hand side tends to zero as $N \to \infty$.
 
 For the integrals against $N_y$ we obtain
 \begin{align*}
     \E&\Big[\Big|\sum_{y\in E\setminus\{x\}}\int_0^t\int_{\X\setminus\{0\}}\int_{\R_+}\nu(x)\mathbbm{1}_{\{u\leq g(y,\eta_{s-}(y))\}}-\nu(x)\mathbbm{1}_{\{u\leq g(y,\eta_{s-}^N(y))\}}\mathbbm{1}_{V_N}(x)N_y(ds,d\nu,du)\Big|\Big]
     \\ &\leq\mathbbm{1}_{V_N^c}(x)\sum_{y\in V\setminus\{x\}}C_3(y)\int_{\X \setminus\{0\}}\nu(x)H_1(y,d\nu)\int_0^t\E[\eta_s(y)]ds
     \\ &\hspace{10pt} + \mathbbm{1}_{V_N}(x)\sum_{y\in V\setminus\{0\}}C_3(y)\int_{\X\setminus\{0\}}\nu(x)H_1(y,d\nu)\int_0^t\E\big[\big|\eta_s(y)-\eta_s^N(y)\big|\big]ds
 \end{align*}
 We estimate both terms separately using (A4). The first one is bounded by
 \begin{align*}
     \mathbbm{1}_{V_N^c}(x)\sum_{y\in E\setminus\{x\}}&C_3(y)\int_{\X \setminus\{0\}}\nu(x)H_1(y,d\nu)\int_0^t\E[\eta_s(y)]ds
     \\ &\leq \mathbbm{1}_{V_N^c}(x)\frac{C_4}{v(x)}\sum_{y\in V \backslash \{x\}}v(y)\int_0^t\E[\eta_s(y)]ds
     \\ &= \mathbbm{1}_{V_N^c}(x)\frac{C_4}{v(x)}\int_0^t\E[\|\eta_s\|]ds<\infty.
 \end{align*}
 Since the last expression is finite, it tends to zero as $N \to \infty$.
 The second term tends to zero as $N \to \infty$ due to \eqref{dominated_convergence} and
 \begin{align*}
     \sum_{y\in V\setminus\{x\}}&C_3(y)\int_{\X \setminus\{0\}}\nu(x)H_1(y,d\nu)\int_0^t\E\big[|\eta_s(y)-\eta_s^N(y)|\big]ds
    \leq\frac{C_4}{v(x)}\int_0^t\E\big[\|\eta_s-\eta_s^N\|]ds.
 \end{align*}
 Finally, for the last integral, we find that 
 \begin{align*}
     \E&\Big[\Big|\int_0^t\int_{\X\setminus\{0\}}\int_{\R_+}\nu(x)\mathbbm{1}_{\{u\leq\rho(x,\eta_{s-},\nu)\}}-\nu(x)\mathbbm{1}_{\{u\leq\rho(x,\eta_{s-}^N,\nu)\}}\mathbbm{1}_{V_N}(x)M(ds,d\nu,du)\Big|\Big]
     \\ &\leq\mathbbm{1}_{V_N^c}(x)\E\Big[\int_0^t\int_{\X\setminus\{0\}}\nu(x)\rho(x,\eta_{s-},\nu)dsH_2(d\nu)\Big]
     \\ &\hspace{10pt}+\mathbbm{1}_{V_N}(x)\E\Big[\int_0^t\int_{\X\setminus\{0\}}\nu(x)|\rho(x,\eta_{s-},\nu)-\rho(x,\eta_{s-}^N,\nu)|ds H_2(d\nu) \Big].
 \end{align*}
 By (A6), the first term is finite (and hence convergent) due to 
 \begin{align*}
     \E\Big[\int_0^t\int_{\X\setminus\{0\}}\nu(x)\rho(x,\eta_{s-},\nu)dsH_2(d\nu)\Big]\leq\frac{C_6}{v(x)}\int_0^t (1+\E\|\eta_s\|)ds<\infty.
 \end{align*}
 For the second term, fix $R>0$. Using (A5) and (A6), we get
 \begin{align*}
     \E&\Big[\int_0^t\int_{\X\setminus\{0\}}\nu(x)|\rho(x,\eta_{s-},\nu)-\rho(x,\eta_{s-}^N,\nu)|dsH_2(d\nu)\Big]
     \\ &\leq\int_0^t\int_{\X\setminus\{0\}}\nu(x)\E\big[\mathbbm{1}_{\{\|\eta_{s-}\|\leq R\}}|\rho(x,\eta_{s-},\nu))-\rho(x,\eta_{s-}^N,\nu)|H_2(d\nu)ds
     \\ &\hspace{10pt}+\int_0^t\E\Big[\int_{\X\setminus\{0\}}\nu(x)\mathbbm{1}_{\{\|\eta_{s-}\|>R\}}|\rho(x,\eta_{s-},\nu)-\rho(x,\eta_{s-}^N,\nu)|H_2(d\nu)\Big]ds
     \\ &\leq\frac{C_5(R)}{v(x)}\int_0^t\E\|\eta_{s-}-\eta_{s-}^N\|_Vds+\int_0^t\E\Big[\int_{\X\setminus\{0\}}\nu(x)\mathbbm{1}_{\{\|\eta_s\|>R\}}\rho(x,\eta_{s-},\nu)H_2(d\nu)\Big]ds
     \\ &\leq\frac{C_5(R)}{v(x)}\int_0^t\E\|\eta_{s-}-\eta_{s-}^N\|ds+\int_0^t\frac{C_6}{v(x)}\E\big[\mathbbm{1}_{\{\|\eta_s\|>R\}}(1+\|\eta_{s-}\|)\big]ds.
 \end{align*}
 As in the calculation for the Brownian part, this estimate implies convergence to zero when $N \to \infty$.
\end{proof}

\section{Proof of Theorem \ref{thm: invariant measures}}

In this section, we prove the existence of an invariant measure, and convergence in the Wasserstein distance towards this measure, i.e., we prove Theorem \ref{thm: invariant measures}.
\begin{proof}[Proof of Theorem \ref{thm: invariant measures}]
 Let $(\eta_t)_{t \geq 0}$ and $(\xi_t)_{t \geq 0}$ be the unique solutions to \eqref{SDE} with deterministic initial conditions $\eta_0,\xi_0 \in \X$ such that $\xi_0 \leq \eta_0$. Then
 $\xi_t \leq \eta_t$ a.s., and hence
 \begin{align*}
     \E[\|\eta_t - \xi_t\|] &= \sum_{x \in V}v(x)\E[(\eta_t(x) - \xi_t(x))]
     \\ &= \sum_{x \in V}v(x) \left( \E[\eta_t(x)] - \E[\xi_t(x)]\right)
     \\ &= \sum_{x \in V}v(x) (\eta_0(x) - \xi_0(x)) 
     + \sum_{x \in V}v(x)\int_0^t \E\left[\widetilde{B}(x,\eta_s(x)) - \widetilde{B}(x,\xi_s(x)\right]ds  
     \\ &\leq \sum_{x \in V}v(x) (\eta_0(x) - \xi_0(x))  - A \int_0^t \left( \sum_{x \in V}v(x)\E[(\eta_s(x) - \xi_s(x))] \right) ds
     \\ &= \| \eta_0 - \xi_0 \| - A \int_0^t \E[ \|\eta_s - \xi_s\|] ds.
 \end{align*}
 The Gronwall lemma yields $\E[\|\eta_t - \xi_t\|] \leq \E[\|\eta_0 - \xi_0\|]e^{-At}$.
 For general deterministic $\xi_0,\eta_0 \in V$ we let $V = \{ x_k \ : \ k \geq 1 \}$ be a numeration of $V$, and define
 \[
  \xi_0^n(x) = \begin{cases} \eta_0(x_k), & k = 1, \dotsc, n
  \\ \xi_0(x_k), & k > n
  \end{cases}
 \]
 with $\xi_0^0 = \xi_0$. Then
 \[
  \xi_0^{n+1}(x_k) - \xi_0^n(x_k) = \begin{cases}
  0, & k \neq n + 1
  \\ \eta_0(x_{n+1}) - \xi_{0}(x_{n+1}), & k = n+1
  \end{cases}
 \]
 and hence
 for each $n \in \N$ either
 $\xi_0^n \leq \xi_0^{n + 1}$ or $\xi_0^{n+1} \leq \xi_0^n$. Let $(\xi_t^n)_{t \geq 0}$ be the unique solution of \eqref{SDE} with initial condition $\xi^n_0$. Previous consideration yields
 \begin{align*}
  \E[\|\xi_t^n - \xi^{n+1}_t\|] &\leq \E[\|\xi_0^{n} - \xi^{n+1}_0\|]e^{-At}
  \\ &= v(x_{n+1})|\eta_0(x_{n+1}) - \xi_0(x_{n+1})| e^{-At}.
 \end{align*}
 Hence we obtain
 \begin{align*}
     \E[ \|\eta_t - \xi_t|] 
     &\leq \sum_{k=0}^{n-1} \E[\|\xi_t^k - \xi_t^{k+1}\|]
     + \E[\|\xi_t^{n} - \eta_t \|]
     \\ &\leq  e^{-At}\sum_{k=0}^{n-1} v(x_{k+1})|\eta_0(x_{k+1}) - \xi_0(x_{k+1})|+ \E[\|\xi_t^{n} - \eta_t \|]
     \\ &\leq e^{-At}\| \eta_0 - \xi_0\| + \E[\|\xi_t^{n} - \eta_t \|].
 \end{align*}
 Since the constants $\sup_{R > 0}(C_1(R) + C_5(R)) < \infty$, Theorem \ref{thm: uniqueness} implies that
 \[
  \E[ \|\xi_t^n - \eta_t\| ] \leq \|\xi_0^n - \eta_0\| e^{ct}
  = \sum_{k=n+1}^{\infty}v(x_k)\eta_0(x_k) e^{ct} \longrightarrow 0, \ \ n \to \infty,
 \]
 where the constant $c$ is independent of $n$.
 Hence we obtain
 \[
  \E[ \|\eta_t - \xi_t\|]  \leq e^{-At}\| \eta_0 - \xi_0\|
 \]
 which readily yields \eqref{eq: wasserstein 1 contraction estimate}.
 Since $(\X,d)$ is a Polish space, $(\mathcal{P}_1(\X),W_1)$ is a Polish space as well (see e.g. \cite[Theorem 6.18]{villani}). Therefore the existence and uniqueness of an invariant measure as well as \eqref{eq: wasserstein 1 inv measure} are immediate consequences of \eqref{eq: wasserstein 1 contraction estimate}.
 This completes the proof.
\end{proof}

\section{Linear speed of spread and growth bound}

In this section, we prove Theorem \ref{thm: at most linear}. We consider $V$ to be the vertex set of an infinite connected graph $G = (V,E)$ of bounded degree. Let $\mathrm{dist}(z,z')$ be the graph distance for $z,z' \in V$, and for $x \in V$ and $r >0$ we define
\[
\mathbb{B}(x,r) : = \{z \in V: \text{dist}(z,x) \leq r\}
\]
as the set of nodes in $V$ that are within the graph distance $r$ from $x$. Denote by $d$ the maximum degree of $G$, that is, the maximum  degree of its vertices. Note that $d \geq 2$ since $G$ is connected. For a given $x \in V$ and $k\in \N$ there are at most $d^k$ distinct nodes $y \in V$ satisfying $\text{dist}(x,y) = k$. Hence
\begin{equation}\label{quietschen = sqeak}
    \#\mathbb{B}(x,r) \leq  1+d+...+d^r \leq d^{r+1} , 
\ \ \ r \in \N. 
\end{equation}
The proof of Theorem \ref{thm: at most linear} relies on a heat kernel estimate that is a direct consequence of \cite[Corollary 12]{HeatKernelGraph} (see also \cite{Pang93}) as formulated below.
\begin{Lemma}\label{verdrehen = twist, distort}
 Let $(S_t, t \geq 0)$ be a
 nearest neighbour continuous-time random walk
 on an infinite connected graph $\widetilde G $ of bounded degree 
 with vertex set $\widetilde V$. The jump rate from $u \in  \widetilde V$ to $v  \in \widetilde V$ is given by $ \beta(u,v) >0$
 if $u \sim v$, and $0$ otherwise. Here $u \sim v$ indicates that $u,v$ are neighbours. Assume that $\sup_{u \sim v} \beta(u,v) < \infty$ and there exists $m > 0$ such that
 \[
  \sum\limits_{v:\ v \sim u} \beta(u,v) \geq m > 0, \qquad \forall u \in \widetilde{V}.
 \]
 Let $K(t,u,v) = \P _u \{ S_t = v \}$ be the transition probability starting from $u$ to be at $v$ at time $t$. Then for $u,v \in \widetilde V$ and $t \geq 0$
 \begin{equation}\label{kauen = chew}
   K(t,u,v) \leq 
      \frac {1}{m} \exp \Big[  -\widetilde d(u,v) \ln \Big(  \frac{2\widetilde d(u,v)}{et} \Big) \Big],
 \end{equation}
 where $\widetilde d$ is the graph distance in $\widetilde G$. 
\end{Lemma}
We note that $K$ is also referred to as the heat kernel. To see that Lemma \ref{verdrehen = twist, distort} does indeed follow from \cite[Corollary 12]{HeatKernelGraph} we take in notation of \cite{HeatKernelGraph} $b(g) =   \beta(g)$, $g \in { \rm{ \widetilde S}}$, $a(u) = \sum_{v:\ v \sim u}\beta(u,v)$, so that $k \equiv 1$. The constant $\Lambda$ is defined in \cite{HeatKernelGraph}
as an infimum of the spectrum of a certain operator
which in our case can be seen as the generator of $(S_t, t \geq 0)$.
The inequalities  $0 \leq \Lambda \leq d-1$ 
for a $d$-regular graph
follow from \cite[Lemma 2]{HeatKernelGraph}, 
which enables us to
drop $\Lambda$ in \eqref{kauen = chew}.

Under the assumptions of Theorem \ref{thm: at most linear}, recall that $\widetilde{B}$ denotes the effective drift defined in \eqref{eq: effective drift}. It is convenient to separate the martingale components from the drift of the process. To this end, we rewrite \eqref{SDE} in such a way that all stochastic integrals become martingales, i.e.
\begin{align}\label{SDE sulky}
    \eta_t(x) &= \eta_0(x) + \int_0^t \widetilde{B}(x,\eta_s)ds + \int_0^t \sqrt{2c(x,\eta_s(x))} dW_s(x) 
    \\ \notag &\ \ \ + \sum_{y \in V}\int_0^t \int_{\X} \int_{\R_+} \nu(x) \1_{\{u \leq g(y,\eta_{s-}(y))\}} \widetilde{N}_{y}(ds,d\nu,du).
\end{align}
where the noise terms are the same as in (N1) -- (N4). 
\begin{Lemma}\label{lemma: compact support}
    Suppose that the conditions of Theorem \ref{thm: at most linear} are satisfied. Then
    \[
     \int_{\X\backslash \{0\}} \nu(x) H_1(y,d\nu) = 0
    \]
    holds for all $x,y \in V$ such that $\mathrm{dist}(x,y) > R$ and $g(y,\cdot) \neq 0$.
\end{Lemma}
\begin{proof}
 Using the particular form of $\widetilde{B}$ combined with \eqref{Wirrsal = confusion}, we obtain for each $x \in V$
 \begin{align*}
  \sum_{y \in V \backslash \{x\}}g(y,\eta(y))\int_{\X \backslash \{0\}}\nu(x)H_1(y,d\nu) - B_1(x,\eta(x)) \leq \widetilde{B}(x,\eta) \leq \sum_{y \in V}b(x,y)\eta(y).  
 \end{align*}
 Take $y \neq x$ arbitrary and $\eta(w) = \e\1_{\{w = y\}}$, then
 \[
  g(y,\e)\int_{\X \backslash \{0\}}\nu(x)H_1(y,d\nu) \leq b(x,y)\e, \qquad \e > 0, \ x,y \in V, \ x \neq y.
 \]
 Now let $x,y \in V$ be such that $\mathrm{dist}(x,y) > R$ and $g(y,\cdot) \neq 0$. Then $b(x,y) = 0$ and hence $g(y,\e)\int_{\X \backslash \{0\}}\nu(x)H_1(y,d\nu)=0$. Since $g(y, \cdot) \neq 0$, we find $\e > 0$ such that $g(y,\e) > 0$, which gives $\int_{\X \backslash \{0\}}\nu(x)H_1(y,d\nu) = 0$.
\end{proof}
We are now prepared to prove the result.
\begin{proof}[Proof of Theorem \ref{thm: at most linear}]
 \textit{Step 1.} As a first step we use a comparison principle to reduce the problem to the case of a constant drift. Let $\xi$ be the unique strong solution of
\begin{align*}
    \xi_t(x) &= \eta_0(x) + \sum_{y \in V}\int_0^t b(x,y)\xi_s(y) ds + \int_0^t \sqrt{2c(x,\xi_s(x))} dW_s(x) 
    \\ \notag &\ \ \ + \sum_{y \in V}\int_0^t \int_{\X} \int_{\R_+} \nu(x) \1_{\{u \leq g(y,\xi_{s-}(y))\}} \widetilde{N}_{y}(ds,d\nu,du), \ \ \ x \in V.
\end{align*}
Since $\sum_{y \in V}b(x,y)\eta(y) \geq \widetilde{B}(x,\eta)$ holds for all $x \in V$ and $\eta \in \X$, the comparison principle implies that $\eta_t(x) \leq \xi_t(x)$ holds a.s.. Hence it suffices to prove the assertion for $\xi_t$. Similarly, letting 
\[
 M := \sup\limits_{x, y \in V} b(x,y) < \infty,
\]
we may consider another process $(\zeta_t)_{t \geq 0}$ defined as the unique strong solution of 
\begin{align*}
 \zeta_t(x) &= \eta_0(x) + M\sum_{y \in V}\int_0^t \1_{\{y: \mathrm{dist}(x,y) \leq R\}}\zeta_s(y) ds + \int_0^t \sqrt{2c(x,\zeta_s(x))} dW_s(x) 
    \\ \notag &\ \ \ + \sum_{y \in V}\int_0^t \int_{\X} \int_{\R_+} \nu(x) \1_{\{u \leq g(y,\zeta_{s-}(y))\}} \widetilde{N}_{y}(ds,d\nu,du), \ \ \ x \in V.
\end{align*}
Since $b(x,y) \leq M \1_{\{\mathrm{dist}(x,y)\leq R\}}$, the comparison principle in Theorem \ref{thm comparison different drift} 
yields a.s. $\zeta_t(x) \geq \xi_t(x)$ for all $t\geq 0$ and $x \in V$. Hence, it suffices to prove the assertion for $\zeta_t$.

 \textit{Step 2.} In this step we derive an estimate of the growth of $\E[\zeta_t(x)]$ with respect to $x \in V$ and show \eqref{der Vorrat = supply}. For this purpose, consider a new graph $\widehat G = (V, \widehat E)$ with vertex set $V$, and $u, v  \in V$ are neighbours in $\widehat E$ if and only if $\text{dist}(x, y) \leq R$. Let $\widehat{d}$ be the graph distance on $\widehat{G}$. Then note that $\widehat d(x,y) \leq \text{dist}(x,y) \leq R \widehat d(x,y)$ holds for all $x,y \in V$. Consider a continuous-time random walk $(S_t, t\geq 0)$ on $\widehat{G}$ with transition rates $q_{x,y} = M \1_{\{\text{dist}(x,y) \leq R \}}$ for $x,y \in V$, and let $K(t,x_0,y)$ be the transition probabilities at time $t$ from $x_0$ to $y$. Applying Lemma \ref{verdrehen = twist, distort} to this random walk on $\widehat G$ and using the equivalence of the graph distances $\widehat{d}$ and $\mathrm{dist}$, we find that
\begin{align*}\label{der Schal}
    \notag K(t,x_0,y) &\leq 
    \frac {1}{ M} \exp \Big[  -\widehat d(x_0,y) \ln \Big(  \frac{2\widehat d(x_0,y)}{et} \Big) \Big]
    \\ &\leq \frac {1}{M} \exp \Big[  -\frac{\text{dist}(x_0,y)}{R} \ln \Big(  \frac{2\text{dist}(x_0,y)}{eRt} \Big) \Big].
\end{align*}
In order to relate this bound to the original process $\zeta_t$, let us define $f_x: \R _+ \longrightarrow \R _+$ with $x \in V$ as the expectation of $\zeta_t(x)$, i.e., $f_x (t) = \E \zeta _t (x)$. Taking expectation in \eqref{SDE sulky} we get the representation
\begin{align*}
    f_t(x) &= f_0(x) + M\sum_{y \in V: \ \mathrm{dist}(x,y) \leq R} \int_0^t f_s(y)ds.
\end{align*}
Now we focus on obtaining an upper bound on $f_x(t)$.
The transition probabilities $K(t,x_0,x)$ of $(S_t, t\geq 0)$ satisfy the equation 
\begin{align*}
\begin{cases}
     \frac{d}{dt} K(t,x_0,x)
     &=  -M d_R(x)  K(t,x_0,x) +
     M \sum\limits_{y \in E: 1\leq \text{dist}(x,y) \leq R} 
     K(t,x_0,y), 
     \\ \ \ \ K(t,x_0,x) &=   \1_{\{ x = x_0\}}.
     \end{cases}
\end{align*}
where $d_R(x) = \#\{y \in V: 1\leq \text{dist}(x,y) \leq R\} = \#\mathbb{B}(x,r) - 1$ is the number of vertices within distance $R$ from $x$. Hence $K(t,x_0,\cdot) = e^{tA}\1_{\{x_0\}}$, where $\1_{\{x_0\}}$ is the indicator function and $A$ is a bounded operator on $L^{\infty}(V)$ defined by 
\[
Ah(x) = - Md_R(x)h(x) + M \sum\limits_{y \in V: 1\leq \text{dist}(x,y) \leq R} h(y), \qquad h \in L^{\infty}(V).
\]
Consider another bounded linear operator $B$ on $L^{\infty}(V)$ defined by 
\[
 Bh(x) = -MD h(x) + M \sum\limits_{y \in V: 1\leq \text{dist}(x,y) \leq R}h(y), \qquad h \in L^{\infty}(V),
\]
where $D$ is the maximum degree of $\widehat G$. Note that by \eqref{quietschen = sqeak}, $D \leq d^{R+1} < \infty$, and that  $f_{\cdot}(t) = e^{ M D t }e^{tB} \1_{\{x_0\}}$. Finally, \eqref{quietschen = sqeak} implies $d_R(x) \leq D$ for $x \in V$, whence we have $A \geq B$ and hence also $e^{tA} \geq e^{tB}$.
This readily yields
$$
K(t,x_0, \cdot) = e^{tA}\1_{\{x_0\}} \geq e^{tB}\1_{\{x_0\}} = e^{ -M D t } f_{\cdot}(t), \ \ \  t \geq 0,
$$
and hence $K(t,x_0,y) \geq f_y(t) e^{ -M D t }$ for all $y \in V$. 
In view of \eqref{v: growth bound 2}, we find constants $C_v, \ell > 0$ such that
 \begin{align}\label{eq: v growth condition 1}
  v(x) \geq C_v D^{-\ell \mathrm{dist}(x_0,x)}, \qquad x \in V.
\end{align} 
Define the constant
\[
 C_0 = \max\left\{ 1,\ \frac{MDR}{\ln(D)},\ \frac{D^{(2\ell + 5)R}eR}{2} \right\}.
\] 
Using the above estimates we obtain for all $y \in V$ and $t > 0$ satisfying $\mathrm{dist}(x_0,y) > C_0 t$ that 
\begin{align}
    \notag f_y(t) &\leq \frac{e^{MDt}}{M}\exp\left[ - \frac{\mathrm{dist}(x_0,y)}{R}\ln\left( \frac{2\mathrm{dist}(x_0,y)}{eRt}\right)\right]
    \\ \notag &\leq \frac{1}{M}\exp\left[ \left(\frac{MD}{C_0} - \frac{1}{R}\ln\left(\frac{2C_0}{eR}\right)\right)\mathrm{dist}(x_0,y)\right]
    \\ \notag &\leq \frac{1}{M}\exp\left[ \left( \frac{MD}{C_0} - \frac{(2\ell+5)\ln(D)}{R}\right)\mathrm{dist}(x_0,y) \right]
    \\ &\leq \frac{D^{-2(\ell+2)\mathrm{dist}(x_0,y)/R}}{M}. \label{eq: fy estimate}
\end{align}

\textit{Step 3.} Next, we prove a similar estimate for the expected supremum of the process, i.e., for $\E\left[ \sup_{s \in [0,t]}\zeta_s(x)\right]$. Let $x \ne x_0$ and $t > 0$. Recall that $\eta_0(x) = 0$. Then, by \eqref{SDE sulky}, we arrive at
\begin{align*}
    \sup_{r \in [0,t]}\zeta_r(x) &\leq M\sum_{y \in V: \ \mathrm{dist}(x,y) \leq R}\int_0^t \zeta_s(y)ds + \sup_{r \in [0,t]}\left|\int_0^r \sqrt{2c(x,\zeta_s(x))} dW_s(x) \right|
    \\ &\ \ \  + \sum_{y \in V \backslash \{x\}}\sup_{r \in [0,t]}\left|\int_0^r \int_{\X}\int_{\R_+} \nu(x) \1_{\{u \leq g(y,\zeta_{s-}(y))\}}\widetilde{N}_y(ds,d\nu,du) \right|.
    \\ &\ \ \  + \sup_{r \in [0,t]}\left|\int_0^r\int_{\X}\int_{\R_+} \nu(x) \1_{\{u \leq g(x,\zeta_{s-}(x))\}}\widetilde{N}_x(ds,d\nu,du) \right|.
\end{align*}
Let us bound all terms in expectation. Doob's maximal inequality applied to the continuous martingale gives by (A2) 
\begin{align*}
\E\left[ \sup_{r \in [0,t]}\left|\int_0^r \sqrt{c(x,\zeta_s(x))}dW_s(x) \right|^2 \right] 
&\leq 4C_2(x) \int_0^t f_s(x)ds
\\ &\leq 4\left(\sum_{y \in V}v(y)C_2(y)\right)\frac{1}{v(x)}\int_0^t f_s(x)ds
\end{align*} 
For the sum against $\widetilde{N}_y$ with $x \neq y$ we obtain from (A3) 
\begin{align*}
  &\ \sum_{y \in V \backslash \{x\}}\E\left[\sup_{r \in [0,t]}\left|\int_0^r \int_{\X \backslash \{0\}}\int_{\R_+} \nu(x) \1_{\{u \leq g(y,\zeta_{s-}(y))\}}\widetilde{N}_y(ds,d\nu,du) \right| \right]
  \\ &\leq 2\sum_{y \in V \backslash \{x\}}\int_0^t \int_{\X \backslash \{0\}} \nu(x) \E[g(y,\zeta_{s-}(y))] H_1(y,d\nu)ds
  \\ &\leq 2\sum_{y \in V \backslash \{x\}}C_3(y) \int_{\X \backslash \{0\}} \nu(x)H_1(y,d\nu) \int_0^t f_y(s)ds.
\end{align*}
Finally, for the integrals against $\widetilde{N}_x$ we consider $\{\|\nu\| \leq 1\} \backslash \{0\}$ and $\{\|\nu\| > 1\}$ separately. Namely, we obtain from the Burkholder-Davis-Gundy inequality for Poisson random measures and then (A3) combined with (A4) 
\begin{align*}
    &\ \E\left[  \sup_{r \in [0,t]}\left| \int_0^r \int_{\{\|\nu\| \leq 1\}\backslash \{0\}}\int_{\R_+} \nu(x) \1_{\{u \leq g(x,\zeta_{s-}(x))\}}\widetilde{N}_x(ds,d\nu,du) \right|^2 \right]
    \\ &\leq 4\int_0^t \int_{\{\|\nu\| \leq 1\}\backslash \{0\}} \nu(x)^2 g(x,\zeta_{s}(x))H_1(x,d\nu)ds
    \\ & \leq 4 C_3(x)\left( \int_{\{\|\nu\| \leq 1\}\backslash \{0\}} \nu(x)^2 H_1(x,d\nu) \right)\int_0^t f_s(x)ds
    \\ &\leq 4 \left(\sum_{y \in V}C_3(y)\left( \int_{\{\|\nu\| \leq 1\}\backslash \{0\}} \nu(y)^2 H_1(y,d\nu) \right) \right)\frac{1}{v(x)}\int_0^t f_s(x)ds, 
\end{align*} 
while for the big jumps we obtain from (A4)
\begin{align*}
    &\ \E\left[  \sup_{r \in [0,t]}\left| \int_0^r \int_{\{\|\nu\| > 1\}}\int_{\R_+} \nu(x) \1_{\{u \leq g(x,\zeta_{s-}(x))\}}\widetilde{N}_x(ds,d\nu,du) \right| \right]
    \\ &\leq 2\E\left[ \int_0^t \int_{\{\|\nu\| > 1\}}\int_{\R_+} \nu(x) \1_{\{u \leq g(x,\zeta_{s-}(x))\}}H_1(x,d\nu)duds \right]
    \\ &= 2\left( \int_{\{\|\nu\| > 1\}} \nu(x) H_1(x,d\nu) \right)\int_0^t \E[g(x,\zeta_{s}(y))]ds
    \\ &\leq 2C_3(x)\left( \int_{\{\|\nu\| > 1\}} \nu(x) H_1(x,d\nu) \right)\int_0^t f_s(x)ds
    \\ &\leq 2C_4 \int_0^t f_s(x)ds.
\end{align*}
Combining all these estimates, we find a constant $C > 0$ such that 
\begin{align*}
\E\left[\sup_{r \in [0,t]} \zeta_r(x) \right] 
&\leq M \sum_{y \in V: \ \mathrm{dist}(x,y) \leq R} \int_0^t f_s(y)ds
 \\ &\qquad + 2\sum_{y \in V \backslash \{x\}}C_3(y)\left( \int_{\X \backslash \{0\}} \nu(x) H_1(y,d\nu) \right)\int_0^t f_s(y)ds
 \\ &\qquad + \frac{C}{\sqrt{v(x)}}\left(\int_0^t f_s(x)ds \right)^{1/2} + 2C_4 \int_0^t f_s(x)ds
\end{align*} 
holds for each $x \in V$.
Let $t_0 > 0$ be arbitrary. Letting $t > t_0$ and $x \in V$ be such that 
\[
 \mathrm{dist}(x_0,x) > \left(C_0 + \frac{R}{t_0}\right)t,
\]
we find for $y \in V$ satisfying $\mathrm{dist}(x,y) \leq R$ that
\[
 d(x_0,y) \geq d(x,x_0) - d(x,y) \geq \left(C_0 + \frac{R}{t_0}\right)t - R > C_0t.
\]
Hence we can use the previously shown inequality \eqref{eq: fy estimate} on $f_y(s)$ for $s \in (0,t]$ from Step 2 to find that
\begin{align*}
    \E\left[\sup_{r \in [0,t]} \zeta_r(x) \right]
    &\leq \sum_{y \in V: \ \mathrm{dist}(x,y) \leq R}t D^{-2(\ell+2)\mathrm{dist}(x_0,y)/R} 
    \\ &\ \ \ + \frac{2}{M}\sum_{y \in V \backslash \{x\}}C_3(y)\left( \int_{\X \backslash \{0\}} \nu(x) H_1(y,d\nu) \right)t D^{-2(\ell+2)\mathrm{dist}(x_0,y)/R}
    \\ &\ \ \ + \frac{C}{\sqrt{M}} \sqrt{\frac{t}{v(x)}}D^{-(\ell+2)\mathrm{dist}(x_0,x)/R} + \frac{2C_4}{M}tD^{-2(\ell+2)\mathrm{dist}(x_0,x)/R}.
\end{align*} 
Since the graph $G = (V,E)$ is connected the maximum degree satisfies $D>1$.
For the first and second terms, we use the elementary inequality $xD^{-x} \leq \frac{1}{e\ln(D)}$, $x > 0$, to find 
\begin{align*}
 tD^{-2(\ell + 2)\mathrm{dist}(x_0,y)/R} &< \frac{\mathrm{dist}(x_0,y) R^{-1}D^{-\mathrm{dist}(x_0,y)/R}}{C_0}RD^{-(2\ell+3)\mathrm{dist}(x_0,y)/R} 
 \\ &\leq \frac{R}{e\ln(D)C_0}D^{-\mathrm{dist}(x_0,y)/R}
 \\ & \leq \frac{RD^R}{e \ln(D)C_0}  D^{- \mathrm{dist}(x_0,x)/R}.
\end{align*}
Similarly, we obtain for the last term
\begin{align*}
 tD^{-2(\ell+2)\mathrm{dist}(x_0,x)/R} \leq \frac{RD^R}{e \ln(D)C_0}  D^{- \mathrm{dist}(x_0,x)/R}.
\end{align*}
For the remaining third term, we use \eqref{eq: v growth condition 1} and the elementary inequality $\sqrt{x}D^{-x} \leq \frac{1}{\sqrt{2e\ln(D)}}$ to find that
\begin{align*}
     \sqrt{\frac{t}{v(x)}} D^{-(\ell+2)\mathrm{dist}(x_0,x)/R} &\leq C_v\frac{\sqrt{ R^{-1}\mathrm{dist}(x_0,x)}D^{-\mathrm{dist}(x_0,x)/R}}{\left(C_0 + \frac{R}{t_0}\right)^{1/2}} R^{1/2} D^{-\mathrm{dist}(x_0,x){/R}} 
     \\ &\leq \frac{R^{1/2}D^{-\mathrm{dist}(x_0,x)/R}}{\sqrt{2e\ln(D)\left(C_0 + \frac{R}{t_0}\right)}}.
\end{align*} 
Thus, combining these estimates gives for $t > t_0$ and some constant $C' > 0$ independent of $x,y,t$ the estimate 
\begin{align*}
 \E\left[\sup_{r \in [0,t]} \zeta_r(x) \right] 
 &\leq C'D^{-\mathrm{dist}(x_0,x)/R}
 \\ &\ \ \ + C' \sum_{y \in V \backslash \{x\}}C_3(y) \int_{\X \backslash \{0\}} \nu(x) H_1(y,d\nu) D^{-\mathrm{dist}(x_0,x)/R} 
 \\ &\leq \left(C' + C_4 d^{R+1}\sup_{\mathrm{dist}(x,y) \leq R}\frac{v(y)}{v(x)}\right) D^{-\mathrm{dist}(x_0,x)/R}
\end{align*}
where we have used Lemma \ref{lemma: compact support} to find 
\begin{align*}
  &\ \sum_{y \in V \backslash \{x\}}C_3(y) \int_{\X \backslash \{0\}} \nu(x) H_1(y,d\nu)
  \\  &\leq \sum_{y: \ \mathrm{dist}(y,x) \leq R,\ y \neq x} \1_{\{C_3(y) > 0\}}\frac{C_3(y)}{v(x)} \int_{\X \backslash \{0\}} \sum_{w \in V \backslash \{y\}}v(w)\nu(w) H_1(y,d\nu)
   \\ &\leq C_4\sum_{y: \ \mathrm{dist}(y,x) \leq R,\ y \neq x} \frac{v(y)}{v(x)}
   \\ &\leq C_4 d^{R+1}\sup_{\mathrm{dist}(x,y) \leq R}\frac{v(y)}{v(x)} < \infty.
\end{align*} 
 and have set, without loss of generality, $C_3(y) = 0$ whenever $g(y,\cdot) = 0$.

\textit{Step 4.} In this last step we derive the assertion from the Borel-Cantelli lemma. Namely, letting $C_1 = \left(C_0 + \frac{R}{t_0}\right)$, we obtain for $t \geq t_0$ and $\e > 0$, the estimate
\begin{align*}
   \sum\limits_{t \in \N} \P \left[\sup\limits_{x \in V: |x| > C_1 t} \ \sup_{r \in [0,t]} \zeta_r (x) > \varepsilon \right]
   &\leq \frac{1}{\e}\sum\limits_{t \in \N} \sum\limits_{ \substack{ x \in V:\\ |x| > C_1 t }} \E\left[\sup_{r \in [0,t]} \zeta_r(x) \right] 
   \\ &\leq \frac{C''}{\e}\sum\limits_{t \in \N} \sum\limits_{ \substack{ x \in V:\\ |x| > C_1 t }} D^{-\mathrm{dist}(x_0,x)/R}.
\end{align*}
Since $D > 1$, the right-hand side is finite and we may apply the Borel-Cantelli lemma. This gives
\begin{equation*}
    \P \left[\sup\limits _{x \in V: |x| > C_1 t} \ \sup_{r \in [0,t]} \zeta_r(x)  > \varepsilon
    \text{ for infinitely many } t \in \N \right] = 0
\end{equation*}
which concludes the proof. 
\end{proof}

\appendix


\section{Martingale property for \eqref{eq: martingale pathwise uniqueness}}

In this section, we will show that $(\mathcal{M}(t \wedge \tau_m))_{t \geq 0}$ is a martingale for each $m,k \geq 1$. For this purpose, we write
$\mathcal{M}(t \wedge \tau_m) = \sum_{j=1}^{5}\mathcal{M}_j(t \wedge \tau_m)$
with
\begin{align*}
    \mathcal{M}_1(t \wedge \tau_m) &:= \int_0^{t \wedge \tau_m} \phi_k'(\zeta_s(x)) \left( \sqrt{2c(x,\eta_s(x))} - \sqrt{2 c(x,\xi_s(x))}\right)dW_s(x),
    \\ \mathcal{M}_2(t \wedge \tau_m) &:= \int_0^{t \wedge \tau_m} \int_{ \{ \| \nu \| \leq 1\} \backslash \{0\}} \int_{\R_+}  D_{\Delta_0(x,s)}\phi_k(\zeta_{s-}(x))  \widetilde{N}_x(ds,d\nu,du),
    \\ \mathcal{M}_3(t \wedge \tau_m) &:= \int_0^{t \wedge \tau_m} \int_{ \{ \| \nu \| > 1\}} \int_{\R_+}  D_{\Delta_0(x,s)}\phi_k(\zeta_{s-}(x))  \widetilde{N}_x(ds,d\nu,du),
    \\ \mathcal{M}_4(t \wedge \tau_m) &:= \sum_{y \in V \backslash \{x\}}\int_0^{t \wedge \tau_m} \int_{ \X \backslash \{0\}} \int_{\R_+}  D_{\Delta_0(y,s)}\phi_k(\zeta_{s-}(x))  \widetilde{N}_y(ds,d\nu,du),
    \\ \mathcal{M}_5(t \wedge \tau_m) &:= \int_0^{t \wedge \tau_m} \int_{\X \backslash \{0\}}\int_{\R_+} D_{\Delta_1(x,s)}\phi_k(\zeta_{s-}(x)) \widetilde{M}(ds,d\nu,du).
\end{align*}
Then it suffices to prove the following lemma:
\begin{Lemma}
 Under the notation of Section 2, the following holds:
\begin{enumerate}
    \item[(a)] $(\mathcal{M}_1(t \wedge \tau_m))_{t \geq 0}$
    is a continuous square-integrable martingale;
    
    \item[(b)] $\left(\mathcal{M}_2(t \wedge \tau_m) \right)_{t \geq 0}$ is a square-integrable martingale;
    
    \item[(c)] $\left(\mathcal{M}_3(t \wedge \tau_m)\right)_{t \geq 0}$ is an integrable martingale;
    
    \item[(d)] $\left(\mathcal{M}_4(t \wedge \tau_m) \right)_{t \geq 0}$ is an integrable martingale.
    
    \item[(e)] $\left( \mathcal{M}_5(t \wedge \tau_m) \right)_{t \geq 0}$ is an integrable martingale.
\end{enumerate}
\end{Lemma}
\begin{proof}
 (a) Using first Ito's isometry, then $| \phi_k' | \leq 1$ and $(a-b)^2 \leq |a^2 - b^2|$ for $a,b \geq 0$, we obtain
 \begin{align*}
     \E\left[ \left| \mathcal{M}_1(t \wedge \tau_m) \right|^2 \right]
     &= \E\left[ \int_0^{t \wedge \tau_m} \phi_k'(\zeta_s(x))^2 \left( \sqrt{2c(x,\eta_s(x))} - \sqrt{2 c(x,\xi_s(x))}\right)^2 ds \right]
     \\ &\leq 2 \E\left[ \int_0^{t \wedge \tau_m} \left|c(x,\eta_{s-}(x)) - c(x,\xi_{s-}(x)) \right| ds \right]
     \\ &\leq 2C_2(x) \E\left[ \int_0^{t \wedge \tau_m} |\eta_{s-}(x) - \xi_{s-}(x)| ds \right]
     \\ &\leq C_2(x)\frac{4mt}{v(x)} < \infty,
 \end{align*}
 where we have used \eqref{eq: boundedness by localization} and (A2).
 
 (b) Recall that $\widehat{N}_x(ds,d\nu,du) = ds H_1(x,d\nu)du$, so that the assertion follows from 
 \begin{align*}
    &\ \E\left[ \int_0^{t \wedge \tau_m} \int_{ \{ \| \nu \| \leq 1\} \backslash \{0\}} \int_{\R_+}  |D_{\Delta_0(x,s)}\phi_k(\zeta_{s-}(x)) |^2 ds H_1(x,d\nu) du \right]
     \\ &\leq \E\left[ \int_0^{t \wedge \tau_m} \int_{ \{ \| \nu \| \leq 1\} \backslash \{0\}} \int_{\R_+} | \Delta_0(x,s) |^2 ds H_1(x,d\nu)du \right]
     \\ &= \E\left[ \int_0^{t \wedge \tau_m} \int_{ \{ \| \nu \| \leq 1\} \backslash \{0\}} \nu(x)^2 | g(x,\eta_{s-}(x)) - g(x,\xi_{s-}(x))| ds H_1(x,d\nu) \right]
     \\ &\leq C_3(x)\left( \int_{ \{ \| \nu \| \leq 1\} \backslash \{0\}} \nu(x)^2 H_1(x,d\nu)\right) \E\left[ \int_0^{t \wedge \tau_m}|\eta_{s-}(x) - \xi_{s-}(x)| ds\right]
     \\ &\leq  2C_3(x)\left( \int_{ \{ \| \nu \| \leq 1\} \backslash \{0\}} \nu(x)^2 H_1(x,d\nu)\right) \frac{mt}{v(x)} < \infty,
 \end{align*}
 where we have used \eqref{eq: boundedness by localization} and (A3), (A4).
 
 (c) Analogously to the estimates in part (b) we obtain
 \begin{align*}
     &\ \E\left[ \int_0^{t \wedge \tau_m} \int_{ \{ \| \nu \| > 1\}} \int_{\R_+} \left| D_{\Delta_0(x,s)}\phi_k(\zeta_{s-}(x)) \right|  ds H_1(x,d\nu)du \right]
     \\ &\leq \E\left[ \int_0^{t \wedge \tau_m} \int_{ \{ \| \nu \| > 1\}} \int_{\R_+} \left| \Delta_0(x, s) \right|  ds H_1(x,d\nu)du \right]
     \\ &= \E\left[ \int_0^{t \wedge \tau_m} \int_{ \{ \| \nu \| > 1\}} \nu(x) \left| g(x,\eta_{s-}(x)) - g(x,\xi_{s-}(x)) \right|  ds H_1(x,d\nu) \right]
     \\ &\leq C_3(x) \left(\int_{ \{ \| \nu \| > 1\}} \nu(x) H_1(x,d\nu) \right) \E\left[ \int_0^{t \wedge \tau_m}|\eta_{s-}(x) - \xi_{s-}(x)| ds \right]
     \\ &\leq 2C_3(x) \left(\int_{ \{ \| \nu \| > 1\}} \nu(x) H_1(x,d\nu) \right) \frac{mt}{v(x)} < \infty.
 \end{align*}
 
 (d) Similarly we obtain in this case
 \begin{align*}
  &\ \E\left[ \left| \mathcal{M}_4(t \wedge \tau_m) \right| \right]
   \\ &\leq \sum_{y \in V \backslash \{x\}} \E\left[ \left| \int_0^{t \wedge \tau_m} \int_{ \X \backslash \{0\}} \int_{\R_+}  D_{\Delta_0(y,s)}\phi_k(\zeta_{s-}(x))  \widetilde{N}_y(ds,d\nu,du) \right| \right]
   \\ &\leq 2\sum_{y \in V \backslash \{x\} }
   \E\left[ \int_0^{t \wedge \tau_m} \int_{ \X \backslash \{0\}} \int_{\R_+}  \left| D_{\Delta_0(y,s)}\phi_k(\zeta_{s-}(x)) \right| ds H_1(y,d\nu)du \right]
  \\ &\leq \frac{2}{v(x)}\sum_{y \in V\backslash \{x\}} C_3(y) \left(\int_{ \X \backslash \{0\}} v(x)\nu(x) H_1(y,d\nu) \right) \E\left[ \int_0^{t \wedge \tau_m}|\eta_{s-}(y) - \xi_{s-}(y)| ds \right]
  \\ &\leq \frac{2}{v(x)} \sum_{z \in V} \sum_{y \in V \backslash \{z\}} C_3(y) \left(\int_{ \X \backslash \{0\}} v(z)\nu(z) H_1(y,d\nu) \right) \E\left[ \int_0^{t \wedge \tau_m}|\eta_{s-}(y) - \xi_{s-}(y)| ds \right]
  \\ &= \frac{2}{v(x)} \sum_{y \in V} C_3(y) \left(\int_{ \X \backslash \{0\}} \sum_{z \in V \backslash \{y\}} v(z)\nu(z) H_1(y,d\nu) \right) \E\left[ \int_0^{t \wedge \tau_m}|\eta_{s-}(y) - \xi_{s-}(y)| ds \right]
  \\ &\leq \frac{2C_4}{v(x)} \sum_{y \in V} V(y) \E\left[ \int_0^{t \wedge \tau_m}|\eta_{s-}(y) - \xi_{s-}(y)| ds \right]
  \\ &\leq \frac{4mtC_4}{v(x)} < \infty,
 \end{align*}
 where we have used (A4).
 Hence the series is absolutely convergent in $L^1$ and thus $(\mathcal{M}_4(t \wedge \tau_m))_{t \geq 0}$ is a martingale.
 
 (e) Using that $\widehat{M}(ds,d\nu,du) = ds H_2(d\nu)du$, the assertion follows from
 \begin{align*}
     &\ \E\left[ \int_0^{t \wedge \tau_m} \int_{\X \backslash \{0\}}\int_{\R_+} \left| D_{\Delta_1(x,s)}\phi_k(\zeta_{s-}(x)) \right| ds H_2(d\nu)du \right]
     \\ &\leq \E \left[ \int_{0}^{t \wedge \tau_m} \int_{\X \backslash \{0\}}\nu(x) | \rho(x,\eta_s, \nu) - \rho(x,\xi_s, \nu) | H_2(d\nu) ds \right]
     \\ &\leq \frac{1}{v(x)} \E\left[ \int_0^{t \wedge \tau_m} \int_{\X \backslash \{0\}} \sum_{y \in V}v(y) \nu(y)| \rho(y,\eta_s, \nu) - \rho(y,\xi_s, \nu)| H_2(d\nu) ds \right]
     \\ &\leq \frac{C_5(m)}{v(x)}\E\left[ \int_0^{t \wedge \tau_m} \| \eta_s - \xi_s \| ds \right]
     \\ &\leq \frac{2m C_5(m)t}{v(x)} < \infty
 \end{align*}
 where we have used (A5).
\end{proof}

\bibliographystyle{amsplain}
\bibliography{references}

\end{document}